\pdfoutput=1
\documentclass[final]{siamart1116}

\usepackage{lipsum}
\usepackage{amsfonts}
\usepackage{graphicx}
\usepackage{epstopdf}
\usepackage{algorithmic}
\ifpdf
  \DeclareGraphicsExtensions{.eps,.pdf,.png,.jpg}
\else
  \DeclareGraphicsExtensions{.eps}
\fi

\numberwithin{theorem}{section}

\newcommand{\TheTitle}{\parbox{0.92\textwidth}{Certifying unstability of Switched Systems using Sum of Squares Programming}}
\newcommand{\TheAuthors}{B. Legat, R. M. Jungers, and P. A. Parrilo}

\headers{\TheTitle}{\TheAuthors}

\title{{\TheTitle}\thanks{Submitted to the editors September 28, 2017.
A preliminary version of this work appeared in Proceedings of Hybrid Systems: Computation and Control, 2016 \cite{legat2016generating}.
}}

\author{
  Beno\^{i}t Legat%
  \thanks{ICTEAM, Universit\'e catholique de Louvain, 4 Av. G. Lema\^itre, 1348 Louvain-la-Neuve, Belgium
    (\email{benoit.legat@uclouvain.be}).}
  \and
  Pablo A. Parrilo%
  \thanks{Laboratory for Information and Decision Systems, Massachusetts Institute of Technology, 77 Massachusetts Avenue, Cambridge MA 02139, USA
    (\email{parrilo@mit.edu}).}
  \and
  Rapha\"{e}l M. Jungers%
  \thanks{ICTEAM, Universit\'e catholique de Louvain, 4 Av. G. Lema\^itre, 1348 Louvain-la-Neuve, Belgium
    (\email{raphael.jungers@uclouvain.be}).}
}

\usepackage{amsopn}

\ifpdf
\hypersetup{
  pdftitle={\TheTitle},
  pdfauthor={\TheAuthors}
}
\fi

\usepackage{custom}
\usepackage{jsr}
\newcommand{\propref}[1]{\cref{prop:#1}}
\newcommand{\theoref}[1]{\cref{theo:#1}}
\newcommand{\lemref}[1]{\cref{lem:#1}}

\newcommand{\cororef}[1]{\cref{coro:#1}}
\newcommand{\defref}[1]{\cref{def:#1}}
\newcommand{\exemref}[1]{\cref{exem:#1}}
\newcommand{\progref}[1]{\cref{prog:#1}}
\newsiamremark{myexem}{Example}
\newsiamremark{myprog}{Program}
\newsiamremark{myrem}{Remark}

\begin{document}

\maketitle

\begin{abstract}
  The joint spectral radius (JSR) of a set of matrices characterizes the
maximal asymptotic growth rate of an infinite product of matrices of
the set.  This quantity appears in a number of applications including
the stability of s\witch{} and hybrid systems. A popular method used for
the stability analysis of these systems searches for a Lyapunov function with convex optimization tools.
We investigate dual formulations for this %convex optimization
approach
and leverage these dual programs for developing new analysis tools for the JSR.

%In this paper,
We show that the dual of this convex problem searches
for the \emph{occupations measures} of trajectories with high asymptotic growth
rate.
We both show how to generate a sequence of guaranteed high asymptotic growth rate and how to detect cases where we can provide lower bounds to the JSR.
%The moment relaxation of this problem is in fact the dual of the
%well known sum of squares optimization program for finding Lyapunov
%functions. Therefore, solving this program pair with a given candidate value
%for the JSR either returns Lyapunov functions certifying that the JSR
%is lower than this candidate value or returns the moments of relaxed
%occupation \emph{pseudo-measures}.
%We give a rounding procedure to extract high
%growth rate trajectories from these moments and provide a guarantee for
%this growth rate.
We deduce from it a new guarantee for the upper bound provided by the sum of squares lyapunov program.
%We can also detect that the
%JSR is larger than the candidate value by finding atomic measures that have
%the required moments.
We end this paper with a method to reduce the computation of the JSR of
low rank matrices to the computation of the constrained JSR of matrices
of small dimension.

All results of this paper are presented for the general case of constrained switched systems,
that is, systems for which the switching signal is constrained by an automaton.
%For this reason, all results of this paper are generalized for constrained
%switched systems.

\end{abstract}

% REQUIRED
\begin{keywords}
  Joint spectral radius, Sum of squares programming,
  S\witch{} Systems, Path-complete Lyapunov functions
\end{keywords}

% REQUIRED
%	93D05  	Lyapunov and other classical stabilities (Lagrange, Poisson, $L^p, l^p$, etc.)
% 93D09  	Robust stability
% 93D20  	Asymptotic stability
% 93D30  	Scalar and vector Lyapunov functions
\begin{AMS}
  93D05, 93D20, 93D30
\end{AMS}

\section{Introduction}
In recent years, the study of the stability of hybrid systems
has been the subject of extensive research using methods based on classical ideas from Lyapunov theory
and modern mathematical optimization techniques.
Even for s\witch{} linear systems, arguably the simplest class of hybrid systems,
determining stability is undecidable and approximating the maximal asymptotic growth rate that a trajectory can have is NP-hard \cite{blondel2000boundedness}.
Despite these negative results, the vast range of applications has motivated a wealth of algorithms to approximate
this maximal asymptotic growth rate.

%\Sect\ref{sec:dualalgo} is given in these more general settings.

A s\witch{} linear system is characterized by a finite set of matrices $\A \eqdef \{A_1, A_2,\\\ldots, A_m\} \subset \mathbb{R}^{n \times n}$ and the iteration
\begin{equation}
  \label{eq:switchsys}
  x_k = A_{\sigma_k} x_{k-1}, \quad \sigma_k \in [m].
\end{equation}
The maximal asymptotic growth rate of this iteration is given by the \emph{joint spectral radius} (JSR).
The JSR $\jsr$ of a finite set of matrices $\A$
is defined as
\[ \jsr = \lim_{k \to \infty} \max_{\sigma \in [m]^k} \|A_{\sigma_k} \cdots A_{\sigma_2} A_{\sigma_1}\|^{1/k}. \]
This definition is independent of the norm used.

The JSR was introduced by Rota and Strang~\cite{rota1960note}
and has many other applications such as wavelets, the capacity of some particular codes, zero-order stability of ordinary differential equations, congestion control in computer networks, curve design and networked and delayed control systems; see \cite{jungers2009joint} for a survey on the JSR and its applications.
Many algorithms exist for estimating the JSR but not much is known on how to generate an infinite
sequence of matrices with an asymptotic growth rate close to the JSR.
However generating such sequence can be of particular interest, depending on the application,
such as exhibiting unstable trajectories for s\witch{} linear systems.
%To the best of our knowledge, the
The
currently known algorithms generate a sequence of matrices with high spectral radius
using brute force (or branch-and-bound variants) and repeat this sequence infinitely \cite{gripenberg1996computing,guglielmi2008algorithm,jungers2014lifted}.

%Arbitrarily tight approximation of
Approximating the JSR usually consists in certifying upper bounds $\gamub$ to the JSR by exhibiting a Lyapunov function or invariant set for the matrices $A_i/\gamub$.
The search for such Lyapunov functions can naturally be written as a convex optimization program; see \progref{primalinf}.
Certifying lower bounds $\gamlb$ is currently either achieved using the guarantees we have on the accuracy of the upper bound to the JSR
or by exhibiting trajectories of asymptotic growth rate $\gamlb$.
In this paper, we introduce a new way to certify lower bounds by exhibiting nonnegative measures satisfying some invariance condition parametrized by the matrices $A_i/\gamlb$; see \eqref{eq:dualinf1}.
% Too precise, concept not introduced, let's say at the level of ideas
%This generalizes the notion of trajectories since given an trajectory with asymptotic growth rate $\gamlb$,
%its \emph{occupation measure} will satisfy this condition for the matrices $A_i/\gamlb$.
This invariance condition is linear on the measure hence the search of measures on the convex cone of nonnegative measures is a \emph{convex} program; see \progref{dualinf}.
It turns out that this program is the dual of \progref{primalinf}.

We revisit the sum-of-squares program proposed by Parrilo and Jadbabaie~\cite{parrilo2008approximation} and show that its dual formulation is the moment relaxation of the search of the measures satisfying the invariance condition.
%Parrilo and Jadbabaie~\cite{parrilo2008approximation} shows how to use sum of squares (SOS) programming to compute Lyapunov functions satisfying \progref{primalinf}, the SOS program is given in \progref{primal}.
%In this paper
%In fact, the moment relaxation of \progref{dualinf}, given in \progref{dual}, is the dual of \progref{primal}.

Thanks to this duality, solving this pair of programs with a given candidate value $\gamma$
for the JSR either returns Lyapunov functions certifying that
$\jsr \leq \gamma$ or returns moments that are solution of the moment relaxation.
These moments are not necessarily the moments of measures satisfying the invariance conditions.
However, we give a rounding procedure to extract a (infinite) switching sequence from these moments
and provide a guarantee on the asymptotic growth rate of this sequence.
%This guarantee also implies a new one for the upper bound provided by the sum of squares lyapunov program.
As a by-product of the rounding procedures, the spectral radius of a finite part of
this infinite sequence can be used to give lower bounds on the JSR.
In addition, we give a way to sometimes detect when the moments belongs to measures that satisfy the invariance conditions.
This happens when the measures are the convex combination of the occupation measures of several periodic trajectories.
Since the trajectories are periodic, the measures are atomic and we can recover them from moments of sufficiently high degree.
We show on numerical examples that these techniques work well in practice.
%This is explained in \secref{lb} with numerical examples suggesting that it works well in practice.

In some applications the values that $\sigma_k$ can take in \eqref{eq:switchsys} may depend on $\sigma_{k-1}, \sigma_{k-2},\\\ldots$.
These constraints are often conveniently represented using a \emph{finite automaton} and the JSR under such constraints
is called \emph{constrained joint spectral radius} (CJSR) \cite{dai2012gel}; an example of constrained s\witch{} system is given by \exemref{run1} and its automaton is illustrated by \figref{run}.

The following will serve as a running example.

\begin{myexem}[Running example\footnote{The source code and instructions are available at the author's web site
to reproduce the numerical results obtained for the running example.
The s\witch{} system considered in Example~\ref{exem:simple1}, Example~\ref{exem:simple2} and Example~\ref{exem:simple3}
is simpler than the running example and the results given in these examples can be obtained by hand.
}]
  \label{exem:run1}
  We borrow the example of \cite[\Sect4]{philippe2016stability}.
  %It is based on a state-feedback control that might undergo dropouts in its state feedback.
  The set of matrices $\A$ is composed of the following four matrices
  \begin{align*}
    A_1 & = A+B
    \begin{pmatrix}
      k_1 & k_2
    \end{pmatrix},&
    A_2 & = A+B
    \begin{pmatrix}
      0 & k_2
    \end{pmatrix},\\
    A_3 & = A+B
    \begin{pmatrix}
      k_1 & 0
    \end{pmatrix},&
    A_4 & = A.
  \end{align*}
% \begin{align*}
%   \A
%   & = \{A_1,A_2,A_3,A_4\}\\
%   & = \{
%     A+B
%     \begin{pmatrix}
%       k_1 & k_2
%     \end{pmatrix},
%     A+B
%     \begin{pmatrix}
%       0 & k_2
%     \end{pmatrix},
%     A+B
%     \begin{pmatrix}
%       k_1 & 0
%     \end{pmatrix},
%     A
%   \}.
%  \end{align*}
  where $k_1 = -0.49$, $k_2 = 0.27$,
  \[
    A =
    \begin{pmatrix}
      0.94 & 0.56\\
      0.14 & 0.46
    \end{pmatrix} \text{ and }
    B =
    \begin{pmatrix}
      0\\
      1
    \end{pmatrix}.
  \]

  The automaton is represented by Figure~\ref{fig:run}.

  \begin{figure}[!ht]
    \centering
    \begin{tikzpicture}[x=40,y=18]
       \Vertex{1}
       \NO[unit=4](1){2}
       \SOEA[unit=2](2){3}
       \EA[unit=2](3){4}
       \tikzset{EdgeStyle/.style = {->}}
       \tikzset{LoopStyle/.style = {->}}
       \tikzset{EdgeStyle/.append style = {bend right=18}}
       \Edge[label=1](1)(3)
       \Edge[label=3](1)(2)
       \Edge[label=2](2)(1)
       \Edge[label=1](2)(3)
       \Edge[label=2](3)(1)
       \Edge[label=3](3)(2)
       \Edge[label=4](3)(4)
       %\Loop[dist=1cm,dir=EA](3)
       \draw[thick,->] (3) to [out=115,in=65,looseness=10] node [midway, fill=white] {1} (3);
       \Edge[label=1](4)(3)
     \end{tikzpicture}
     \caption{Automaton for the running example.
     The numbers on the \arcs{} are their respective labels.
     %We can see that $(2,3)$ is $\G$-admissible
     %while $(3,4)$ is not.
     }
     \label{fig:run}
  \end{figure}
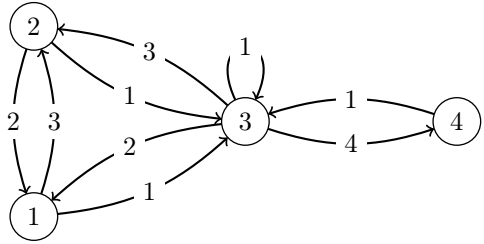
\end{myexem}

The automaton representing the constraints can be represented by a
strongly connected labelled directed graph $\G(\Nodes,\Arcs)$,
possibly with parallel \arcs{}.  The labels are elements of the set
$[m]$ and $\Arcs$ is a subset of $\Nodes \times \Nodes
\times [m]$.  We say that $(u,v,\sigma) \in \Arcs$ if
there is an \arc{} between node $u$ and node $v$ with label $\sigma$.

\begin{equation}
  \label{eq:cswitchsys}
  x_k = A_{\sigma_k} x_{k-1}, \quad (\sigma_1, \ldots, \sigma_k) \in \G_k.
\end{equation}

The \emph{arbitrary switching} case~\eqref{eq:switchsys}
can be seen as the particular case when the automaton has only one node and $m$ self-loops with labels $1, \ldots, m$.
On the other side, any constrained switched system can be replaced by an arbitrary switching system with the same CJSR; see \lemref{koz}.
Our techniques are well suited for analysing the more general constrained systems as well.

We also provide
a new estimate of the accuracy of the SOS-based approximation algorithm for the CJSR which is better
than the previously existing one for sufficiently large SOS degree.
The existing estimate only depends on the dimension of the matrices
while our new one relates the accuracy of the SOS-based approximation algorithm with the combinatorial structure of the automaton representing the constraints.
%Two guarantees of accuracy exist for the SOS-based approximation algorithm but only one of them has been generalized in the constrained case yet.
%The generalized guarantee only depends on the dimension of the matrices.
%In \Sect\ref{sec:quality}, we give a generalization for the second guarantee relating the accuracy of the SOS-based approximation algorithm with the combinatorial structure of the automaton representing the constraints.

In \cite{ahmadi2012joint}, Ahmadi and Parrilo show how to reduce the computation of the JSR
of low rank matrices to a combinatorial problem, the CJSR of $1 \times 1$ matrices (i.e. scalars).
As a final contribution, we generalize
this approach and give a reduction of the computation of the JSR (or CJSR)
of rank $r$ matrices to the computation of the CJSR of $r \times r$ matrices.

The paper is organized as follows.
In \Sect\ref{sec:inf}, we give the program searching for Lyapunov functions,
the program searching for measures satisfying the invariance condition,
prove the duality between the two programs
and show that they respectively provides upper and lower bounds to the CJSR.

In \Sect\ref{sec:primal}, we give the SOS program searching for Lyapunov functions and
we give our new estimate for its accuracy.
The new bounds explicitly depend on the allowable transitions, through
the graph $G(V,E)$.

In \secref{dualalgo}, we give the moment relaxation of the program searching for moment measures satisfying the invariance condition.
We both show how to generate the sequence of high asymptotic growth rate and detect cases where
we can provide lower bounds to the CJSR as mentioned in this introduction.

In \Sect\ref{sec:lowrank}, we give the low rank reduction mentioned above.

\paragraph{Notations}
We define the automaton $\G^\Tr (\Nodes,\Arcs^\Tr )$ where
\(\Arcs^\Tr  = \{\, (v,u,\sigma) : (u,v,\sigma)\\\in \Arcs \,\}.\)

We denote as $\Arcs_k$ the subset of $\Arcs^k$ that represents valid paths of length $k$.
The \ktup{k} $(\sigma_1, \sigma_2, \ldots, \sigma_k)$ is said to be $\G$-\emph{admissible} if $\sigma_1, \ldots, \sigma_k$ are the respective labels of a path of length $k$.
The \emph{arbitrary switching} case,
that is, when every \tup{} is $\G$-admissible,
can be seen as the particular case when the automaton has only one node and $m$ self-loops with labels $1, \ldots, m$.
% TODO Tu mets trop de petites phrases 'self-standing', qui font que ton texte devient très haché, et pas très smooth. ex: le paragraphe 'the arbitrary switching case,...'
We denote the set $\{1, \ldots, m\}$ as $[m]$ and the set of all \ktups{k} of $[m]^k$ that are $\G$-admissible as $\G_k$.
The sequence $\sigma_1, \sigma_2, \ldots$ is $\G$-admissible (resp. $\G^\Tr $-admissible) if $(\sigma_1, \ldots, \sigma_k)$ (resp. $(\sigma_k, \ldots, \sigma_1)$) is $\G$-admissible for any $k \geq 1$.
We denote $A_{\sigma_k}\cdots A_{\sigma_1}$ as $A_s$ where $s = (\sigma_1, \ldots, \sigma_k)$ or $s$ is a path with these respective labels.

To shorten the notation we denote the $i$th node of a path $s$ as
$s(i)$ and the $i$th edge as $s[i]$.
Also, for a given \ktup{k} $s$, we denote $(s(i),\ldots,s(k))$ by $s(i:)$.
We define
\begin{align*}
  \Arcs_k(u,v) & = \{\, s \in \Arcsin_k \mid s(1) = u, s(k+1) = v \,\}\\
  \Arcsin_k(v) & = \{\, s \in \Arcsin_k \mid s(k+1) = v \,\}\\
  \Arcsout_k(v) & = \{\, s \in \Arcsout_k \mid s(1) = v \,\}\\
  \Arcsin_k[e] & = \{\, s \in \Arcsin_k \mid s[k] = e \,\}\\
  \Arcsout_k[e] & = \{\, s \in \Arcsout_k \mid s[1] = e \,\}.
\end{align*}
We denote the indegree (resp. outdegree) of a node $v \in \Nodes$ as
$\din(v)$ (resp. $\dout(v)$) and the maximum indegree (resp.
outdegree) of $\G$ as $\mdin(\G) = \max_{v \in \Nodes} \din(v)$ (resp.
$\mdout(\G) = \max_{v \in \Nodes} \dout(v)$).
We also denote the number of paths of length $k$ ending (resp.
starting) at a node $v \in \Nodes$ as $\din_k(v) \eqdef |\Arcsin_k(v)|$ (resp. $\dout_k(v) \eqdef |\Arcsout_k(v)|$)
and define $\mdin_k(\G) = \max_{v \in \Nodes} \din_k(v)$ and
$\mdout_k(\G) = \max_{v \in \Nodes} \dout_k(v)$.
Note that $\mdin_1(\G) = \mdin(\G)$, $\mdout_1(\G) = \mdout(\G)$ and
for any $k$, $\mdout_k(\G^\Tr) = \mdin_k(\G)$.
% TODO mdin vs mdin(\G)

\section{Instability certificate using measures}
%\section{Lyapunov dual functions}
\label{sec:inf}
The definition of the JSR is generalized as follows for constrained systems.
\begin{definition}[\cite{dai2012gel}]
  \label{def:cjsr}
  The \emph{constrained joint spectral radius} (CJSR) of
  \defAG{}, denoted as $\cjsr$,
  is
  \begin{equation}
    \label{eq:jsrtheo}
    \limsup_{k \to \infty} \cjsrrk = \cjsr = \lim_{k \to \infty} \cjsrk
  \end{equation}
  where
  \begin{equation}
    \label{eq:cjsrrk}
    \cjsrrk
    = \max \big\{\, \jsrc : c \in \G_k, c \text{ is a cycle}\,\big\}, \quad \jsrc = [\rho(A_c)]^{1/k},
  \end{equation}
  and
  \begin{equation}
    \label{eq:cjsrk}
    %\cjsrk = \max \Big\{ \|A_{\sigma_k} \cdots A_{\sigma_2} A_{\sigma_1}\|^{1/k}\\: (\sigma_1, \sigma_2, \ldots, \sigma_k) \text{ is }\G\text{-admissible} \Big\}.
    \cjsrk =
    \max \big\{\, \|A_s\|^{1/k}\\: s \in \G_k \,\big\}.
  \end{equation}
\end{definition}
%When talking about the JSR, it is implicit that it is unconstrained.
%We sometimes make it explicit by saying ``unconstrained JSR'' to prevent confusion.

We can readily see that
\begin{equation*}
  \cjsrrk \leq \cjsrk
\end{equation*}
for any $k$ and norm $\norm{\cdot}$.
Equality \eqref{eq:jsrtheo} is called the \emph{Joint Spectral Radius Theorem} and was proved in 1992 by Berger and Wang~\cite{berger1992bounded} in the unconstrained case.
Elsner~\cite{elsner1995generalized} provided a somewhat simpler self contained proof in 1995.
Both proofs use rather involved results on the joint spectral radius.

A popular method for proving stability of a dynamical system is to find a Lyapunov function.
In this section, we introduce measures playing a role dual to Lyapunov function for switched system.
These measures provide a certificate for instability.
Finding Lyapunov functions and finding these measures are in fact two dual programs, %\emph{strongly dual} programs,
they are respectively provided by \progref{primalinf} and \progref{dualinf}.
%The Joint Spectral Radius Theorem is a consequence of the duality between these two programs.
We will be succinct in our definition of measure-theoretic concepts but the interested reader can find an good introduction to writing programs using measures and functions as decision variables in \cite{lasserre2009moments}.

Consider the dual pair $(\Fb, \Fbd)$ where $\Fb$ is the space of bounded measurable
functions on $\Sn$ and $\Fbd$ is the space of \emph{finite}\footnote{The measure $\mu$ is \emph{finite} if $\mu(\Sn)$ is finite.} \emph{signed}\footnote{A \emph{signed} measure is a difference between two measures, i.e. $\mu - \nu$ where $\mu$ and $\nu$ are measures is a signed measure.} Borel measures on $\Sn$. % TODO define finite
Given a function $f(x) \in \Fb$, we can define the homogeneous\footnote{A function $f$ is homogeneous if $f(\alpha x) = \alpha f(x)$ for any scalar value $\alpha$.}
function $h(f) \eqdef x \mapsto \|x\|_2f(x/\|x\|_2)$ on $\R^n$.
We define $\F = \{\, h(f) \mid f \in \Fb \,\}$ with the scalar product $\la h(f), \mu \ra = \la f, \mu \ra$ for $f \in \Fb, \mu \in \Fbd$.

Given an application $A$ and a measure $\mu \in \Fbd$,
the \emph{pushforward measure} $\pushf{A}{\mu}$ is often defined to be the measure given by
$(\pushf{A}{\mu})(B) = \mu(A^{-1}(B))$ for $B \in \Sn$.
However, since $\Sn$ may not be invariant under application of the matrices of $\A$,
we will use an alternative definition.
Given an application $A$ and a measure $\mu$,
the pushforward measure $\pushf{A}{\mu}$ is defined to be the measure such that
$\la f, \pushf{A}{\mu} \ra = \la f \circ A, \mu \ra$ for any $f \in \F$.
Moreover, given $B \subseteq \Sn$, we define $\mu(B) = \la x \mapsto \|x\|_2 \Ind{B}(x), \mu \ra$
so that $(\pushf{A}{\mu})(B)$ is well defined.
Using these definition, one can verify that for any application $A$, measure $\mu \in \Fbd$ and set $B \subseteq \Sn$,
\begin{equation}
  \label{eq:pushfineq}
  \pushf{A}{\mu}(B) \leq \mu(B) \max_{x \in B} \|Ax\|_2.
\end{equation}
%where $\supp(\mu)$ is the support of $\mu$.

% FIXME Do I need to say that it \Fpp = \inte\Fp ?
Let $\Fp$ (resp. $\Fbp$) be the set of nonnegative functions of $\F$ (resp. $\Fb$),
$\Fbpd$ be the set of (nonnegative) measures of $\Fbd$
and $\Fpp$ be the set of positive functions of $\F$.
%The set $\Fp$ is a closed convex cone and its dual is the set of linear functionals that give a nonnegative value when applied to nonnegative functions.
%That is, the set of measures on $\Sn$ which is also a convex cone.
Given two functions $f, g \in \F$, $f \geq 0$ denotes $f \in \Fp$ and $f \geq g$ denotes $f - g \in \Fp$.
Similarly, given two measures $\mu, \nu \in \Fbd$, $\mu \geq 0$ denotes $\mu \in \Fbd$ and $\mu \geq \nu$ denotes $\mu - \nu \in \Fbpd$.

\begin{myprog}[Primal]
  \label{prog:primalinf}
  \begin{align}
    %\label{eq:sosprog}
    \notag
    \inf_{f_v \in \F, \gamub \in \R} \gamub\\
    \label{eq:primalinf1}
    f_v(A_\sigma x) & \leq \gamub f_u(x), \quad \forall (u, v, \sigma) \in \Arcs,\\
    \notag
    f_v(x) & \in \Fpp, \quad \forall v \in \Nodes,\\
    \label{eq:primalinf3}
    \sum_{v \in V} \int_{\Sn} f_v(x) \dif x & = 1.
  \end{align}
\end{myprog}

%The dual of this program is:
\begin{myprog}[Dual of \progref{primalinf}]
  \label{prog:dualinf}
  \begin{align}
%   \notag
%   \sup_{\Exp_{uv\sigma} \in \Fbd, \gamlb \in \R} \gamlb\\
%   \label{eq:dualinf1}
%   \sum_{(u,v,\sigma) \in \Arcs} \mapp_\sigma^* \Exp_{uv\sigma} & \geq \gamlb \sum_{(v,w,\sigma) \in \Arcs} \Exp_{vw\sigma}, \quad \forall v \in \Nodes,\\
%   \sum_{(u,v,\sigma) \in \Arcs} \Exp_{uv\sigma}[f(A_\sigma x)] & \geq \gamlb \sum_{(v,w,\sigma) \in \Arcs} \Exp_{vw\sigma}[f(x)], \quad \forall f \in \Fp, v \in \Nodes,\\
%   \Exp_{uv\sigma} & \in \Fp^*, \quad \forall (u, v, \sigma) \in \Arcs,\\
%   \label{eq:dualinf3}
%   \sum_{(u,v,\sigma) \in \Arcs} \Exp_{uv\sigma}[\|x\|] & = 1
    \notag
    \sup_{\mu_{uv\sigma} \in \Fbd, \gamlb \in \R} \gamlb\\
    \label{eq:dualinf1}
    \sum_{(u,v,\sigma) \in \Arcs} \pushf{A_{\sigma}}{\mu_{uv\sigma}} & \geq \gamlb \sum_{(v,w,\sigma) \in \Arcs} \mu_{vw\sigma}, \quad \forall v \in \Nodes,\\
    \notag
    \mu_{uv\sigma} & \in \Fbpd, \quad \forall (u, v, \sigma) \in \Arcs,\\
    \label{eq:dualinf3}
    \sum_{(u,v,\sigma) \in \Arcs} \mu_{uv\sigma}(\Sn) & = 1.
  \end{align}
\end{myprog}

The constraint \eqref{eq:primalinf1} is the Lyapunov constraint.
The constraint \eqref{eq:dualinf1} is similar to the \emph{measure invariance constraint} $\pushf{A}{\mu} = \mu$ of a linear dynamical system $x_{k+1} = Ax_k$
and to the \emph{mass balance constraint} of a \emph{circulation problem}~\cite{ahuja1993network}.
Without constraint \eqref{eq:primalinf3} (resp. \eqref{eq:dualinf3}), the feasible set of \progref{primalinf} (resp. \progref{dualinf}) is a cone.
These constraints have no effect on the optimal objective value but they make the feasible set bounded.

The main result of this section is summarized in the following theorem.

\begin{theorem}
  \label{theo:exact}
  Consider \defAG{}.
    Let $\gamubopt$ (resp. $\gamlbopt$) be the optimal value of \progref{primalinf} (resp. \progref{dualinf}).
  The following identity holds:
  \[ \gamlbopt = \cjsr = \gamubopt. \]
\end{theorem}

As a consequence of \theoref{exact}, we have a new criterion for lower bounds on the CJSR using measures.
\begin{corollary}
  \label{coro:instabcert}
  Consider \defAGe{}.
  If there exist non-trivial\footnote{At least one $\mu_{uv\sigma}$ must be nonzero.} measures $\mu_{uv\sigma}$ for each $(u,v,\sigma) \in \Arcs$ such that
  \[ \sum_{(u,v,\sigma) \in \Arcs} \pushf{A_{\sigma}}{\mu_{uv\sigma}} \geq \gamlb \sum_{(v,w,\sigma) \in \Arcs} \mu_{vw\sigma}, \quad \forall v \in \Nodes \]
  then
  \( \gamlb \leq \cjsr. \)
\end{corollary}

Since $\gamlbopt = \gamubopt$ by \lemref{duality}, one could prove \theoref{exact} by using that $\cjsr = \gamubopt$ which is classical; see \lemref{primalbuild} and \theoref{stabcert}.
However, to illustrate the relation between atomic solutions of \progref{dualinf} and periodic trajectories, we instead prove that $\gamubopt \leq \cjsrk$ and $\cjsrrk \leq \gamlbopt$ for all $k$.
%We will show that $\gamubopt = \lim_{k\to\infty} \cjsrk$ and $\gamlbopt = \lim_{k\to\infty} \cjsrrk$.
These relations somehow suggest that \progref{primalinf} is related to the definition of the CJSR with norms while \progref{dualinf} is related to the definition of the CJSR with the spectral radius.
%This is the relation between the Joint Spectral Radius Theorem and the duality between \progref{primalinf} and \progref{dualinf}.

%We start with the proofs of $\gamubopt \leq \cjsrk$ (resp. $\cjsrrk \leq \gamlbopt$) for any $k$.
%They consist in building a feasible solution for the problem from the maximizer of $\cjsrk$ (resp. minimizer of $\cjsrrk$).
\begin{lemma}
    \label{lem:primalbuild}
    Consider \defAGe{}.
    For any natural number $k$ and norm $\|\cdot\|$, we have
    \[ \gamubopt \leq \cjsrk \]
    where $\cjsrk$ is defined in \eqref{eq:cjsrk}.
    \begin{proof}
        Let $f_v(x) = \max_{s \in \Arcsout_{k-1}(v)} \|A_s x\|$.
        For any \arc{} $(u,v,\sigma) \in \Arcs$,
        \[ f_v(A_\sigma x) = \max_{s \in \Arcsout_{k-1}(v)} \|A_sA_\sigma x\| \leq \max_{s \in \G_k} \|A_s x\| \leq [\cjsrk]^k \|x\|. \]
        so the Lyapunov functions $f_v$ are solution for $\gamub = \cjsrk$.
    \end{proof}
\end{lemma}

%The feasible solution for \progref{dualinf} is built by \algoref{traj2meas}.
%It works with any measure $\nu_x$ parametrized by $x \in \Sn$ but the case $\nu_x = \delta_x$
%where $\delta_x$ are dirac measure centered at $x$ may be more intuitive.

\begin{lemma}
    \label{lem:traj2meas}
    Consider \defAG{}
    and a cycle $c = (\sigma_1, \ldots, \sigma_k)$ of length $k$ with intermediary \nodes{} $v_0, \ldots, v_{k-1},v_k=v_0 \in \Nodes$ such that $(v_{i-1}, v_i, \sigma_i) \in \Arcs$ for $i = 1, \ldots, k$.
    Let $x_0$ be such that $A_c x_0 = \lambda x_0$ with $|\lambda| = \rho(A_c)$ and $\enorm{x_0} = 1$,
    consider the following iteration
    \begin{align*}
      x_i & = A_{\sigma_i} x_{i-1} &
      \hat{x}_i & = x_i/\enorm{x_i} &
      \alpha_i & = \enorm{x_i}/\jsrc^i
    \end{align*}
    where $\jsrc$ is defined in \eqref{eq:cjsrrk}.
    The following solution
    \[ \Big(\,\mu_{uv\sigma} = \sum_{i=1,v_i = v}^k \alpha_i \dirac{\hat{x}_i} \,\Big)_{(u,v,\sigma) \in \Arcs} \]
    is feasible for \progref{dualinf} with any $\gamlb \geq \jsrc$
    and it satisfies the constraints \eqref{eq:dualinf1} as equality for $\gamlb = \jsrc$.
    \begin{proof}
        By construction, $\alpha_k = 1$ so $\alpha_k\dirac{\hat{x}_k} = \dirac{x_0}$ and
        for each $i = 0, \ldots, k-1$, we have
        \[
            %A_{\sigma_i*}\mu_{v_{i-1}v_i\sigma_i} =
            \pushf{A_{\sigma_i}}{(\alpha_i \dirac{\hat{x}_i})}
            = \alpha_i \frac{\enorm{x_{i+1}}}{\enorm{x_i}} \dirac{\hat{x}_i}
            = \jsrc \alpha_{i+1} \dirac{\hat{x}_{i+1}}
            \leq \gamlb \alpha_{i+1} \dirac{\hat{x}_{i+1}}
            %= \gamma \mu_{v_iv_{i+1}\sigma_{i+1}}
        \]
        which equality if $\jsrc = \gamlb$.
    \end{proof}
\end{lemma}

\begin{lemma}
    \label{lem:dualbuild}
    Consider \defAGe{}.
    For any natural number $k$, we have
    \[ \cjsrrk \leq \gamlbopt \]
    where $\cjsrrk$ is defined in \eqref{eq:cjsrrk}.
    \begin{proof}
        Let $\opt{c} \in \argmax\big\{\, \jsrc : c \in \G_k, c \text{ is a cycle}\,\big\}$,
        by \lemref{traj2meas} we can build a feasible solution of \progref{dualinf} with $\gamlb = \jsrc$.
        %using \algoref{traj2meas}.
    \end{proof}
\end{lemma}

In some sense, \lemref{traj2meas} is encoding a trajectory in the measures $\mu_{uv\sigma}$.
%If $\nu_x = \delta_x$ then
We say that the resulting measures are the \emph{occupation measures} of the trajectory $x_0, x_1, \ldots, x_k$ as defined in \lemref{traj2meas}.

%   It is well known that a Lyapunov function for a switched system provides a certificate for stability, this is shown by \theoref{stabcert}.
%   In this section, we introduce measures playing a role dual to Lyapunov function for switched system.
%   These measures provide a certificate for instability, this is shown by \theoref{instabcert} and
%   are solutions of the dual of the program for finding Lyapunov function.
%   To support this claim, we give the two dual programs and show that they are strongly dual in \theoref{duality}.
%   As a side note, we can see that the Joint Spectral Radius Theorem is a corollary of \theoref{stabcert}, \theoref{instabcert} and \theoref{duality} for which the proof does not require any involved result on the joint spectral radius.
%   It can also be noted that none of the proof of the 3 theorems uses any result on the advanced result on the JSR such as the existence of an extremal norm.

%where $\mapp$ is the linear mapping such that $\mapp f(x) = f(Ax)$.
%Hence its adjoint mapping $\mapp^*$ is such that $\mapp^*\Exp[f(x)] = \la\mapp^*\Exp, f(x)\ra = \la \Exp, \mapp f(x) \ra = \la \Exp, f(Ax) \ra$.

\begin{myexem}
  \label{exem:simple0}
  Consider the unconstrained system \cite[Example~2.1]{ahmadi2012joint} with $m=2$:
  \[ \A = \{ A_1 = e_1e_2^\Tr , A_2 = e_2e_1^\Tr  \} \]
  where $e_i$ denotes the $i$th canonical basis vector.

  A solution to \progref{primalinf} is given by
  \[ (f(x), \gamub) = (\enorm{x}, 1). \]
% Indeed, for example, with $A_1$ we have
% \[ \|e_1e_2^\Tr x\|_2 = \|(x_2, 0, 0)\|_2 = \sqrt{x_2^2 + 0} \leq \sqrt{x_1^2 + x_2^2}. \]
  This means that $f(x)$ is a Lyapunov function for the system so as it is well known
  %, and shown in \theoref{stabcert} for completeness,
  this certifies that $\jsr \leq 1$.

  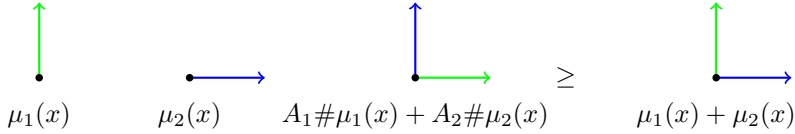
\begin{figure}[!ht]
    \centering
    \begin{tikzpicture}
      \begin{scope}[shift={(-7,0)}]
        \draw[thick,->,green] (0,0) to (0,1);
        \fill (0,0) circle(.05);
        \node at (0,-.5) {$\mu_1(x)$};
      \end{scope}

      \begin{scope}[shift={(-5,0)}]
        \draw[thick,->,blue] (0,0) to (1,0);
        \fill (0,0) circle(.05);
        \node at (0,-.5) {$\mu_2(x)$};
      \end{scope}

      \begin{scope}[shift={(-2,0)}]
        \draw[thick,->,green] (0,0) to (1,0);
        \draw[thick,->,blue] (0,0) to (0,1);
        \fill (0,0) circle(.05);
        \node at (0,-.5) {$\pushf{A_{1}}{\mu_1(x)} + \pushf{A_{2}}{\mu_2(x)}$};
      \end{scope}

      \begin{scope}[shift={(0,0)}]
        \node at (0,0) {$\geq$};
      \end{scope}

      \begin{scope}[shift={(2,0)}]
        \draw[thick,->,green] (0,0) to (0,1);
        \draw[thick,->,blue] (0,0) to (1,0);
        \fill (0,0) circle(.05);
        \node at (0,-.5) {$\mu_1(x) + \mu_2(x)$};
      \end{scope}
    \end{tikzpicture}
    \caption{A representation of the optimal dual solution of \exemref{simple0} with the constraint \eqref{eq:dualinf1}.}
    \label{fig:simple0}
  \end{figure}

  A dual solution $\mu_1$ (resp. $\mu_2$)\footnote{In the arbitrary switching case, we write $\mu_\sigma$ instead of $\mu_{uv\sigma}$ for short}
  for the first (resp. second) matrix has the measure $\mu_1 = \delta_{(0,1)}/2$ (resp. $\mu_2 = \delta_{(1,0)}/2$).
  This is the solution obtained by applying \lemref{traj2meas} to the cycle $(1, 2)$.
  This is shown in \figref{simple0}.

\end{myexem}

\begin{proof}[Proof of \theoref{exact}]
    By \lemref{primalbuild}, \lemref{dualbuild} and \eqref{eq:jsrtheo},
    we know that $\gamubopt \leq \cjsr \leq \gamlbopt$.
    By weak duality between \progref{primalinf} and \progref{dualinf},
    proved in \lemref{duality},
    $\gamlbopt \leq \gamubopt$.
    Therefore equality holds.
\end{proof}

\begin{myrem}
  Occupation measures for continuous switched systems are studied in \cite{claeys2016modal}.
  These measures are supported on the cartesian product of the state space and a finite interval of time $t \in [0, T]$
  while in this paper, the measures are only supported on the subset $\Sn$ of the state space.
  Indeed, since the system \eqref{eq:switchsys} is homogeneous and time-invariant, we can encode trajectories in a measure on $\Sn$ (\lemref{traj2meas}) and
  still be able to recover it (\cororef{instabcert}).

  The measures studied in \cite{hare2011explicit} are supported on the paths in $\G$.
  They are related to the measures studied in this paper since given a cycle $c$,
  we can compute the measures of the trajectory using this switching cycle and the eigenvector of $A_c$ with \lemref{traj2meas}.
\end{myrem}

%   We have shown with \lemref{primalbuild} and \lemref{dualbuild} that $\gamubopt \leq \cjsrk$ and $\cjsrrk \leq \gamlbopt$ for any $k$.
%   By the Joint Spectral Radius Theorem, this means that $\gamlbopt = \gamubopt$.
%   In \lemref{duality}, we provide a direct proof of the duality between the two programs.
%   A direct proof of the inequality $\lim_{k \to \infty} \cjsrk \leq \gamubopt$ is classical
%   and is provided in \theoref{stabcert} for completeness.

%We have seen with \lemref{traj2meas} how to build the occupation measure of a cycle $c$. % a feasible solution of \progref{dualinf} with $\gamlb = \gamma$ from a cycle $c$ with $\rho(A_c)^{1/k} = \gamma$, its occupation measure.
%We are now seeking for a method to do the opposite, recover a cycle from its occupation measure.

One may wonder whether \lemref{traj2meas} also works in the reverse direction to give a \emph{constructive} proof for \cororef{instabcert} when the measures $\mu_{uv\sigma}$ are atomic.
Namely, can we extract a periodic trajectory with $\jsrc \geq \gamlb$ from any atomic feasible solution of \progref{dualinf} with $\gamlb$.
As such solution may be the convex hull of solutions obtained by the construction of \lemref{traj2meas}, we may recover several periodic trajectory, from which there might be only one that satisfies $\jsrc \geq \gamlb$.
The following Lemma provides a constructive way to recover a periodic trajectory $c$ satisfying $\jsrc \geq \gamlb$ in the scalar case%
\footnote{Note that in this case, any measure is atomic since $\Sn$ is zero-dimensional}%
, i.e. $n = 1$%

\begin{lemma}
  \label{lem:instabcerts}
  Consider \defAGs{}.
  If there exists a feasible solution $\mu$ of \progref{dualinf} with $\gamlb$,
  then there exists a cycle $c$ of length $k$ with $\jsrc \geq \gamlb$.
  \begin{proof}
    Let $(\mu, \gamma)$ be the solution.
    By \eqref{eq:dualinf3} and \eqref{eq:dualinf1}, we can find a cycle $c$ for which each edge $e$ has a nonzero measure $\mu_e$.

    If $\jsrc \geq \gamma$, we are done.
    Otherwise, if $\jsrc < \gamma$, using \lemref{traj2meas},
    we can build a feasible solution $\nu$ such that \eqref{eq:dualinf1} is satisfied with equality for $\gamlb = \jsrc$.
    This means that $\mu - \lambda \nu$ is feasible with $\gamma$ for any $\lambda \geq 0$ such that $\mu - \lambda \nu \geq 0$.
    Let $\opt{\lambda}$ be the maximum value of $\lambda$ such that $\max_{\lambda} \mu - \lambda \nu \geq 0$.
    Since $n = 1$, $\Sn$ is zero dimensional so for at leas one edge $e$ of the cycle $c$, $\mu_e - \opt{\lambda} \nu_e$ is zero.
    Moreover, since $\mu_e$ is nonzero for all edge $e$ of the cycle, $\lambda > 0$.
    Therefore, the number of edge with nonzero measure has decreased and at least one of the constraints \eqref{eq:dualinf1} is now satisfied with strict inequality.

    This process can only be repeated finitely many times until we have a zero measure since the number of edges with nonzero measure decrease each time.
    Moreover we will have $\jsrc \geq \gamma$ at least once since the constraints \eqref{eq:dualinf1} cannot be satisfied with strict inequality for the zero measure.
  \end{proof}
\end{lemma}

Given a feasible solution of \progref{dualinf} and a common partition of the support of the measures,
we show in \propref{Bsc} how to transform the solution into a solution of a scalar switched system.
Using this transformation, we can always recover a cycle $c$ for which $\jsrc = \gamma$ from a solution of \progref{dualinf} with $\gamlb = \gamma$ for which the measures are atomic.

\begin{proposition}
  \label{prop:Bsc}
  Consider \defAGe.
  Suppose that there exists a feasible solution $\mu$ of \progref{dualinf} with $\gamlb = \gamma$
  and a finite family $\mathcal{S}$ of disjoint subsets of $\Sn$ such that
  the support of each measure is included in the union of the sets of the family $\mathcal{S}$.
  Then there exists sets $B_1, \ldots, B_k \in \mathcal{S}$ and a cycle $\sigma_1, \ldots, \sigma_k$ of $\G$ such that
  \[ \prod_{i=1}^k \max_{x \in B_i} \|A_{\sigma_i} x\|_2 \geq \gamma^k \]
  and $A_{\sigma_i} B_i \cap B_{i+1} \neq \emptyset$ for $i=1, \ldots, k$ where $B_{k+1}=B_1$.
% where $L$ is the maximum over $\gamma$ of the Lipschitz constant of the function $a_\sigma(x) = A_\sigma x / \|A_\sigma x\|$.
  \begin{proof}
%   By assumption, we know that the elements of $\mathcal{S}$ are disjoint and have diameters less than or equal to $\epsilon$.
    Given a set $B \in \mathcal{S}$ and an \arc{} $e \in \Arcs$, let $\mu_e^B$ denote the measure defined as $\mu_e^B(C) = \mu_e(C \cap B)$.
    We consider a new constrained switched system with matrices $\A' \subseteq \R^{1 \times 1}$ and automaton $\G'(\Nodes', \Arcs')$
    where $\Nodes' = \{\, (v,B) \mid v \in \Nodes, B \in \mathcal{S} \,\}$,
    $e'((u,v,\sigma), B, C) = ((u,B), (v,C), (\sigma,B))$,
    $\Arcs' = \{ e'(e, B, C) \mid e \in \Arcs, B, C \in \mathcal{S}, A_e B \cap C \neq \emptyset \}$,
    and $A'_{(\sigma,B)} = \max_{x \in B} \|A_\sigma x\|_2$.
    From any solution $\mu$ of the original system feasible for $\gamlb$, the following solution of the system with matrices $\A'$ and automaton $\G'$
    \[ \mu'_{e'(e,B,C)} = \frac{(\pushf{A_e}{\mu_e^B})(C)}{(\pushf{A_e}{\mu_e})(B)} \mu_e(B) \]
    is also feasible with $\gamlb$.
    Indeed, by construction, for any $v \in \Nodes, B, C \in \mathcal{S}$, we have
    \begin{align}
        \label{eq:lhstransfer}
        \sum_{e \in \Arcsin_{1}(v), B \in \mathcal{S}} \pushf{A_e}{\mu_e^B}(C)
        & =
        \sum_{e \in \Arcsin_{1}(v), B \in \mathcal{S}} \mu_{e'(e, B, C)}' \frac{(\pushf{A_e}{\mu_e})(B)}{\mu_e(B)}
        \stackrel{\eqref{eq:pushfineq}}{\leq}
        \sum_{e \in \Arcsin[\Arcs']_1(v, C)} \pushf{A'_{e}}{\mu_{e}'}\\
        \label{eq:rhstransfer}
        \sum_{e \in \Arcsout_1(v)} \mu_e(C)
        & =
        \sum_{e \in \Arcsout_1(v), D \in \mathcal{S}} \frac{(\pushf{A_e}{\mu_e^C})(D)}{(\pushf{A_e}{\mu_e})(C)} \mu_e(C)
        =
        \sum_{e \in \Arcsout[\Arcs']_1(v, C)} \mu_{e}'.
    \end{align}
    By \eqref{eq:dualinf1} on $\mu$, the left-hand side of \eqref{eq:rhstransfer} is smaller than
    the left-hand side of \eqref{eq:lhstransfer}.
    Therefore, the right-hand side of \eqref{eq:rhstransfer} is smaller than the right-hand side of \eqref{eq:lhstransfer}
    hence $\mu'$ satisfies \eqref{eq:dualinf1} on the new switched system.

    Therefore, by \lemref{instabcerts},
    there is a cycle $(\sigma_1, B_1), \ldots, (\sigma_k, B_k)$ of $\G'$ such that the modes $\sigma_i$ and sets $B_i$ are as required.
    %\[ \prod_{i=1}^k \max_{x \in B_{i-1}} \|A_{\sigma_i} x\|_2 \geq \gamlb^k \]
    %and $A_{\sigma_i} B_{i-1} \cap B_i \neq \emptyset$ where $B_0 = B_k$.
%   By construction, $\sigma_1, \ldots, \sigma_k$ is also a cycle of $\G$,
%   there exists $\tilde{x}_1 \in B_0, \ldots, \tilde{x}_k$ such that $\tilde{x}_i \in B_{i-1}$ and $A_{\sigma_i}\tilde{x}_i \in B_i$
%   and there exists $\opt{x}_1, \ldots, \opt{x}_k$ such that $\opt{x}_i \in B_{i-1}$ and $\|A_{\sigma_i} \opt{x}_i\|_2 = A'_{(\sigma_i,B_i)}$.
%   We choose some $\hat{x}_0 \in B_0$, and define $\hat{x}_i = a(\hat{x}_{i-1})$ for $i = 1, \ldots k$.
%
%   We see that
%   \[ \dist{\hat{x}_i}{B_i} \leq \|a(\hat{x}_i) - a(\hat{x}_i)\|_2 \leq L\|\hat{x}_i - \hat{x}_i\|_2 \leq L(\dist{\hat{x}_{i-1}}{B_{i-1}} + \epsilon). \]
%   Therefore, $\dist{\hat{x}_i}{B_i} \leq \epsilon L \frac{L^i-1}{L-1}$.
%   Using this inequality, we obtain
%   \[ \|A_{\sigma_{i}}\hat{x}_{i-1}\|_2 \geq \|A_{\sigma_i}\opt{x}_i\| - M\|\opt{x}_i - \hat{x}_{i-1}\|_2 \geq \|A_{\sigma_i}\opt{x}_i\| - \epsilon M\left[L \frac{L^{i-1}-1}{L-1} + 1\right] \]
  \end{proof}
\end{proposition}

\begin{myexem}
  \label{exem:simple05}
  Consider the dual solution obtained in \exemref{simple0}.

  The supports of $\mu_1$ and $\mu_2$ are respectively $B_1 = \{(0, 1)\}$ and $B_2 = \{(1, 0)\}$.
  The automaton $G'(V', E')$ obtained by the transformation of \propref{Bsc} is defined by
  $V' = \{(1, B_1), (1, B_2)\}$ and $E' = \{((1, B_1), (1, B_2), (1, B_1)), ((1, B_2), (1, B_1), (2, B_2))\}$.
  The new $1 \times 1$ matrices are $A'_{(1, B_1)} = 1$ and $A'_{(2, B_2)} = 1$.

  The computation of the CJSR of this scalar system is a \emph{maximum cycle mean} problem
  as outlined in \cite{ahmadi2012joint}.
  The cycle of maximum geometric mean is $((1, B_1), (2, B_2))$ which geometric mean $\sqrt{1 \cdot 1} = 1$.
  We recover the cycle $(1, 2)$ found in \exemref{simple0}.
\end{myexem}

\begin{myrem}
    For some systems, the \emph{finiteness property} does not hold, that is,
    there is no finite cycle $c$ for which $\jsrc$ equals the CJSR.
    For these system, the optimal solutions of \progref{dual} cannot be atomic.
    Can \propref{Bsc} be used to provide a constructive proof of \cororef{instabcert} in this case ?

%   However, the solutions of \progref{dualinf} may also be non-atomic.
    For an arbitrarily small $\epsilon$, using \propref{Bsc} with $\mathcal{S}$ equal to an $\epsilon$-covering of $\Sn$,
    we obtain a cycle $c$ but instead of a eigenvector, we have an ``eigenset'' $B_1$.
    Since the length of the cycle can be large, $A_c B_1$ can have a large diameter too.
    It is not clear how to obtain an eigenvector from this.
    %Hence it is not clear how to build a constructive proof of \theoref{instabcert} using \propref{Bsc}.
    %Hence it seems that the direct relation between \progref{dualinf} and the definition of the JSR with the spectral radius
    %would be with a slightly different definition of the spectral radius using ``eigensets'' instead of eigenvectors.
    %It remains to show that these two definitions are equivalent.
\end{myrem}

\section{Automaton-dependent bounds}
In this section, we introduce a method to approximate the CJSR using SOS programming and provide a new guarantee relating its accuracy with the spectral radius of the adjacency matrix and the $p$-radius; the definition of the $p$-radius can be found in the \appref{pradius}.
\label{sec:primal}

\subsection{Sum of squares programming}
\label{sec:sosprog}
Deciding whether a multivariate polynomial of degree $2d \geq 4$ is nonnegative is known to be NP-hard.
However a sufficient condition for a polynomial to be nonnegative is easy to check.
We say that a polynomial is a \emph{sum of squares} (SOS) if there exist
polynomials $q_1, \ldots, q_M$ such that
\[ p(x) = \sum_{k=1}^Mq_k^2(x). \]
If a polynomial is SOS, then it is obviously nonnegative.

It is well known that if $p(x)$ is an homogeneous polynomial of degree
$2d$ then each $q_k(x)$ must be an homogeneous polynomial of degree $d$;
this can be shown easily using the Newton polytope of $p(x)$ and
\cite[Theorem~1]{reznick1978extremal}.  Let $x^{[d]}$ represent a
basis of the homogeneous polynomials of degree $d$.  We can check
whether a polynomial is SOS using semidefinite programming thanks to
the following theorem.
\begin{theorem}[\cite{choi1995sums,nesterov2000squared,parrilo2000structured,parrilo2003semidefinite,shor1987class}]
  \label{theo:sos}
  A homogeneous multivariate polynomial $p(x)$ of degree $2d$ is a sum
  of squares if and only if
  \[ p(x) = (x^{[d]})^\Tr Q x^{[d]} \]
  where $Q$ is a symmetric positive semidefinite matrix.
\end{theorem}

From the exact arithmetic viewpoint, the basis $x^{[d]}$ chosen in \theoref{sos} does
not affect whether $p(x)$ is SOS or not.  A specific choice of basis
may however improve the numerical behaviour of the corresponding
semidefinite program.

We denote the set of homogeneous polynomials of degree $2d$ as
$\Rh[x]$, the cone of homogeneous SOS polynomials of degree $2d$ as
$\Sosh$ and the dual of $\Sosh$ as $\Soshd$.

A common interpretation of the dual space $\Rhd$ of linear functionals on homogeneous polynomials of degree $2d$ is the space of moments of momonials of degree $2d$.
If $p(x) = a^\Tr x^{[d]}$ and $m$ is the vector of moments of $x^{[d]}$ of a measure $\mu$ then
\[ \la m, a \ra = \int p(x) \dif \mu = \la \mu, p \ra. \]
As a sum of squares polynomial is nonnegative,
this integral is nonnegative for any measure.
Therefore, given a moment vector $m$, a necessary condition for a measure to exist with these moments is that $\la m, a \ra \geq 0$ for any vector of coefficients $a$ of a sum of squares polynomial.
That is, $\Soshd$ is a superset of the set of moments of measures.
The members of $\Soshd$ are often called \emph{pseudo-measures} and denoted $\pmu$; see~\cite{barak2012hypercontractivity}.
%Using this interpretation, $\Soshd$ is the set of moments such that for any sum of square polynomial a

\subsection{CJSR Approximation via SOS}
\label{sec:sos}
The $2d$th root of homogeneous polynomials of degree $2d$ can be used as Lyapunov function.
\begin{proposition}
  \label{prop:pos}
  Consider \defAGe{}.
  If there exist $|\Nodes|$ strictly positive homogeneous polynomials $p_v(x)$ of degree $2d$
  such that
  \[ p_{v}(A_\sigma x) \leq \gamub^{2d}p_{u}(x) \]
  holds for all \arc{} $(u,v,\sigma) \in \Arcs$.
  Then $\cjsr \leq \gamub$.
  \begin{proof}
    Define $f_v(x) = [p_v(x)]^{\frac{1}{2d}}$ and use \theoref{stabcert}.
  \end{proof}
\end{proposition}

We relax the positivity condition of \propref{pos} by the more tractable sum of squares (SOS) condition and
define $\jsrsos$ as the solution of the following SOS restriction of \progref{primalinf}.
\begin{myprog}[Primal]
  \label{prog:primal}
  \begin{align}
    %\label{eq:sosprog}
    \notag
    \inf_{p_v(x) \in \Rh[x],\gamub \in \R} \gamub\\
    %\label{eq:soscons2}
    \notag
    \gamub^{2d} p_u(x) - p_v(A_\sigma x) & \text{ is SOS}, \quad \forall (u, v, \sigma) \in \Arcs,\\
    \label{eq:soscons1}
    p_v(x) & \text{ is SOS}, \quad \forall v \in \Nodes,\\
    \label{eq:soscons3}
    p_v(x) & \text{ is strictly positive}, \quad \forall v \in \Nodes,\\
    \notag
    \sum_{v \in V} \int_{\mathbb{S}^{n-1}} p_v(x) \dif x & = 1.
  \end{align}
\end{myprog}

\begin{myrem}
  In practice we can replace \eqref{eq:soscons1} and \eqref{eq:soscons3} by ``$p_v(x) - \epsilon\|x\|_2^{2d}$ is SOS''
  for any $\epsilon > 0$.
%  It is not necessary for $\epsilon$ to be small
% since without \eqref{eq:soscons3} the feasible set is a cone.
  This constrains $p_v(x)$ to be in the interior of the SOS cone,
  which is sufficient for $p_v(x)$ to be strictly positive.
  The bounds given in \Sect\ref{sec:quality} are valid
  if $p_v(x)$ is in the interior of the SOS cone.
  %A typical choice is $\epsilon = 1$.
  %However this is not required since in practice, SOS solvers, relying on interior point methods, always return solutions which are strictly positive (see...). % Does Does not work
\end{myrem}

%Suppose we have chosen one of the approach of Remark~\ref{rem:strict} to ensure the strict positivity of $p_v(x)$ for all $v \in \Nodes$.
By \propref{pos},
a feasible solution of Program~\ref{prog:primal} gives an upper bound for $\cjsr$,
and thus, for any positive degree $2d$,
\begin{equation}
  \label{eq:upperbound}
  \cjsr \leq \jsrsos.
\end{equation}

%When does there exists nonnegative homogeneous polynomials $p_1(x), \ldots, p_{|\Nodes|}(x)$ that are not all strictly positive
%but are feasible solutions of the SOS program~\eqref{eq:sosprog} ?

\begin{myexem}
  \label{exem:simple1}
  Consider the unconstrained system \cite[Example~2.1]{ahmadi2012joint} with $m=3$:
  \[ \A = \{ A_1 = e_1e_2^\Tr , A_2 = e_2e_3^\Tr , A_3 = e_3e_1^\Tr  \} \]
  where $e_i$ denotes the $i$th canonical basis vector.

  For any $d$, a solution to Program~\ref{prog:primal} is given by
  \[ (p(x), \gamma) = (x_1^{2d} + x_2^{2d} + x_3^{2d}, 1). \]
% Indeed, for example, with $A_1$ we have
% \[ p(e_1e_2^\Tr x) = p(x_2, 0, 0) = x_2^{2d} + 0 + 0 \leq x_1^{2d} + x_2^{2d} + x_3^{2d}. \]
\end{myexem}

\begin{myexem}
  \label{exem:run2}
  Let us reconsider our running example; see Example~\ref{exem:run1}.
  The optimal solution of Program~\ref{prog:primal}
  is represented by
  Figure~\ref{fig:primal} for $2d = 2$, 4, 6, 8, 10 and 12.
%  We can see a big difference between the shape of the sublevel sets
%  for $2d = 2$ and $2d = 4$
%  while between $2d=4$ and $2d=12$,
%  the difference seems to be more subtle.
  \begin{figure}[!ht]
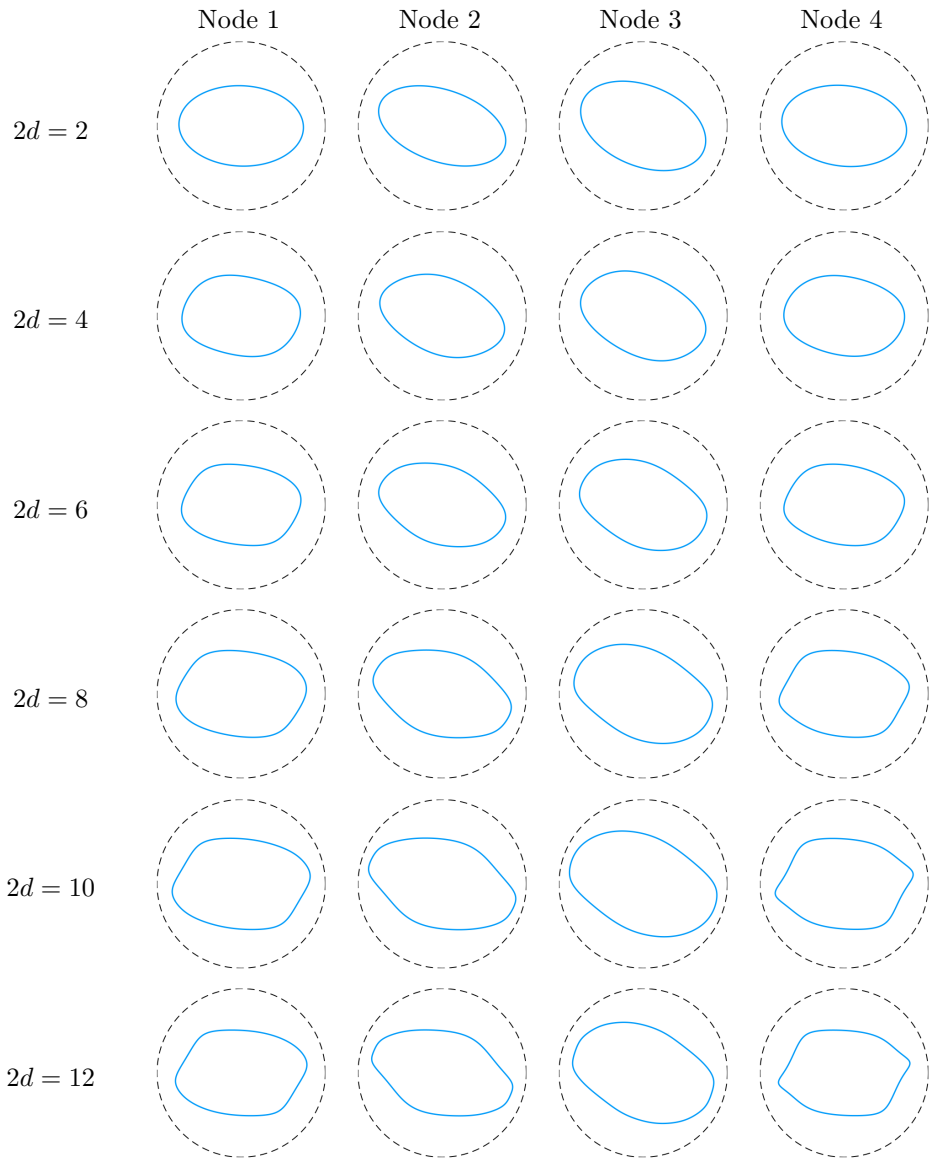

    \centering
    \tabulinesep=0.5mm
    \begin{tabu}{{X[0.85,c,m]X[c,m]X[c,m]X[c,m]X[c,m]}}
      & Node 1 & Node 2 & Node 3 & Node 4\\
      \PrimalFig{2}{1}
      \PrimalFig{4}{2}
      \PrimalFig{6}{3}
      \PrimalFig{8}{4}
      \PrimalFig{10}{5}
      \PrimalFig{12}{6}
    \end{tabu}
    \caption{
    Representation of the solutions to Program~\ref{prog:primal}
    with different values of $d$ for the running example.
    The blue curve represents the boundary of the
    1-sublevel set of the optimal
    solution $p_v$ at each node $v \in \Nodes$.
    The dashed curve is the boundary of the unit circle.
    Observe that some sets are not convex.}
    \label{fig:primal}
  \end{figure}
\end{myexem}

%\begin{myrem}
%  \label{rem:run}
%  \color{red}
%  The shapes reported in \figref{primal} are different than the ones reported in \cite{legat2016generating}.
%  This is due to a different implementation of \progref{primal}.
%  Using the notation of \cite{parrilo2008approximation}, we compute $\rho_{\mathrm{SOS},2d}$ in this paper and we compute $\rho_{\mathrm{CQ},2d}$ in \cite{legat2016generating}.
%\end{myrem}

%\todo{In case of irreducibility (as defined in \cite{philippe2016stability}), we can replace the inf by a min in \progref{primal} and we can say that there is no solution with polynomial that is nonzero+nonnegative but not positive. Should we add that ?}

\subsection{Approximation guarantees}
\label{sec:quality}
In this section, we provide a new bound that relates the accuracy of Program~\ref{prog:primal} to the $p$-radius and the spectral radius of the adjacency matrix of the automaton.
The $p$-radius is a generalization of the joint spectral radius;
we recover the JSR with $p=\infty$.
The definition of the $p$-radius and its property can be found in \appref{pradius}; see \cite{jungers2011fast} for an introduction.

An important property of the $p$-radius is that it is increasing in $p$.
\begin{lemma}
  \label{lem:incp}
  Consider \defAG{}.
  For any integers $p \leq q$,
  \begin{equation}
    \label{eq:incp}
    \cpr[p] \leq \cpr[q] \leq \cjsr \leq \rho(A(\G))^{\frac{1}{q}}
    \cpr[q] \leq \rho(A(\G))^{\frac{1}{p}} \cpr[p].
  \end{equation}
\end{lemma}
The proof can be found in \appref{pradius}.
This Lemma is already known in the unconstrained case~\cite{zhou2002p}.

\begin{myrem}
  \lemref{incp} shows that the $p$-radius provides an upper and lower bound to the CJSR.
  It is known that the $2d$-radius can be computed
  either by computing a spectral radius~\cite{blondel2005computationally}
  or by solving a linear program~\cite{parrilo2008approximation}
  (see \cite{ogura2016efficient} for computation algorithms when $p$ is not an even integer).
  We need to do either of these two operations on
  either the sum of the $d$th Kronecker power of the veronese 2-lift of the matrices
  or the sum of the veronese $2d$-lift of the matrices.
  However, the size of the matrices obtained by taking the $d$th Kronecker power grows rapidly
  and the veronese lifting needs to compute permanents which is very computationally demanding.
  Usually, for the same $d$, \progref{primal} gives approximations of the CJSR with a much higher accuracy at the cost of solving an SDP on smaller matrices.
  While solving an SDP is more demanding than computing a spectral radius or solving a linear program, since the matrices are smaller than with the Kronecker powering
  and we do not need to compute permanents, this operation this method is actually more scalable.
  % TODO ref to the computation section
\end{myrem}

In this section, we show the following bound stating that the solution found by \progref{primal} is at least as good as the bound obtained by computing the $2d$-radius.
\begin{theorem}
  \label{theo:second}
  Consider \defAG{} and a positive integer $d$.
  The approximation given by \progref{primal} using homogeneous polynomials of degree $2d$
  satisfies:
  \begin{equation}
    \label{eq:second}
    \jsrsos \leq \rho(A(\G))^{\frac{1}{2d}}\cdr \leq \rho(A(\G))^{\frac{1}{2d}} \cjsr
  \end{equation}
  where $A(\G)$ is the adjacency matrix of $\G$.
\end{theorem}
Note that the second inequality in \eqref{eq:second} is simply \eqref{eq:incp}.
This theorem is proved at the end of this section.

As a corollary of Theorem~\ref{theo:second}, in the trivial cases such that $\rho(A(\G)) = 1$, the approximation is exact.
This corresponds to the case where every node of $\G$ has indegree and outdegree 1.
In that case, the graph forms a cycle of some length $k$ and the CJSR is simply the $k$th root of the spectral radius of the product of the matrices along this cycle.
%It is reassuring to see that the SOS method is exact in this trivial case.

In the general case, the following approximation guarantee is known (note that the bound does not take into account the particular structure of the automaton):
%Based on \cite{parrilo2008approximation}, the following accuracy guarantees for Program~\ref{prog:primal} is shown in \cite{philippe2016stability}.
\begin{theorem}[{\cite[Theorem~3.6]{philippe2016stability}}]
  \label{theo:matt}
  Consider \defAGn{} and a positive integer $d$.
  The approximation $\jsrsos$ given by Program~\ref{prog:primal} using homogeneous polynomials of degree $2d$
  satisfies:
  \[ \jsrsos \leq {n+d-1 \choose d}^{\frac{1}{2d}} \cjsr. \]
\end{theorem}

The results of \theoref{second}, \theoref{matt} and \eqref{eq:upperbound} are summarized by the following corollary.
\begin{corollary}
  \label{coro:mix}
  Consider \defAGn{} and a positive integer $d$,
  the approximation given by Program~\ref{prog:primal} using homogeneous polynomials of degree $2d$
  satisfies:
  %\begin{multline*}
  \[
    \min\Big\{{n+d-1 \choose d}, \rho(A(\G))\Big\}^{-\frac{1}{2d}} \jsrsos%\\
    \leq \cjsr \leq \jsrsos.
  \]
  %\end{multline*}
  where $A(G)$ is the adjacency matrix of the automaton $\G$.
\end{corollary}
We see that we can have arbitrary accuracy by increasing $d$.
%Another way to get arbitrary accuracy is the ``$T$-Product lift'', see \cite[\Sect3.1]{philippe2016stability}.

For the arbitrary switching case, $\rho(A(\G))$ is equal to the number of matrices $m$.
Theorem~\ref{theo:second} was already known in this particular case \cite[Theorem~4.3]{parrilo2008approximation}.
%This section can thus be seen as a generalization of \cite[\Sect4]{parrilo2008approximation}.
%The remaining of this section is dedicated to the proof of Theorem~\ref{theo:second}
%which is inspired from \cite[\Sect4]{parrilo2008approximation}.

Our proof technique relies on the analysis of an iteration in the vector space of polynomials of degree $2d$.
When this iteration converges, it converges to a feasible solution of Program~\ref{prog:primal}.
By analysing this iteration as affine iterations in this vector space, we derive a sufficient condition for its convergence
and thus an upper bound for $\jsrsos$.

%We now give another bound on the quality of the approximation that only depends on $|\Arcs|$.
Consider the iteration
\begin{align}
  \notag
  p_{v,0}(x) & = 0,\\
  \label{eq:iter}
  p_{v,k+1}(x) & = q_v(x) + \frac{1}{\iterc} %\sum_{u \in \Nodes}
  \sum_{(u,v,\sigma) \in \Arcs} p_{u,k}(A_\sigma x), \quad  v \in \Nodes
\end{align}
for fixed homogeneous polynomials $q_v(x)$ of degree $2d$ in $n$ variables (not necessarily different)
and a constant $\iterc > 0$.

When this iteration converges, it converges to a feasible solution of Program~\ref{prog:primal}.
\begin{lemma}
  \label{lem:conv}
  Consider a constant $\iterc > 0$.
  If there exist homogeneous polynomials $q_v(x)$ in the interior of the SOS cone such that
  iteration~\eqref{eq:iter} converges then
  \[
    \jsrsos \leq \iterc^{\frac{1}{2d}}.
  \]
  \begin{proof}
    Suppose the iteration converges to the polynomials $p_{v,\infty}(x)$.
    It is easy to show by induction that $p_{v,k}(x)$ is SOS for all $k$.
    It is trivial for $k = 0$ and if it is true for $k$ then it is also true for $k+1$ by \eqref{eq:iter}.
    Since the SOS cone is closed, $p_{v,\infty}$ is SOS.
    Now by \eqref{eq:iter}, for each $v \in \Nodes$,
    \[ p_{v,\infty}(x) = q_v(x) + \frac{1}{\iterc} %\sum_{u \in \Nodes}
    \sum_{(u,v,\sigma) \in \Arcs} p_{u,\infty}(A_\sigma x) \]
    so $p_{v,\infty}(x)$ is also in the interior of the SOS cone.
    For each edge $(u,v,\sigma)$, by manipulating the above equation, we have
    \[ \iterc p_{v,\infty}(x) - p_{u,\infty}(A_\sigma x) = \iterc q_v(x) + %\sum_{u \in \Nodes}
    \sum_{(u',v',\sigma') \in \Arcs} p_{u',\infty}(A_{\sigma'} x) \]
    so $\iterc p_{v,\infty}(x) - p_{u,\infty}(A_\sigma x)$ is SOS.
    Therefore $(\{\, p_{v,\infty}(x) : v \in \Nodes \,\}, \iterc^{\frac{1}{2d}})$ is a feasible solution of Program~\ref{prog:primal}.
  \end{proof}
\end{lemma}

In view of \lemref{conv}, it is thus natural to analyse
under which condition iteration~\ref{eq:iter} converges. Recall
that iteration~\ref{eq:iter} is an affine map on the vector space of
homogeneous polynomials of degree $2d$.

\begin{proof}[Proof of Theorem~\ref{theo:second}]
  Iteration~\ref{eq:iter} is an affine map on the vector space of
  homogeneous polynomials of degree $2d$.  It is well known that
  if the convergence is guaranteed when we only retain the linear
  part of the affine map then it is also guaranteed for the affine
  iteration.

  Therefore we can analyse instead the following iteration
  \begin{align*}
    p_{v,0}(x) & = q_v(x),\\
    p_{v,k+1}(x) & = \frac{1}{\iterc} %\sum_{u \in \Nodes}
    \sum_{(u,v,\sigma) \in \Arcs} p_{u,k}(A_\sigma x), \quad  v \in \Nodes
  \end{align*}

  We can see that
  \begin{equation}
    p_{v,k+1}(x) = \frac{1}{\iterc^k}\sum_{s \in \Arcsin_k(v)} q_{s(1)}(A_s x)\\
    \leq \frac{1}{\iterc^k}\sum_{s \in \Arcsin_k(v)} q_{s(1)}(A_s x).
  \end{equation}

  Consider a norm $\|\cdot\|$ of $\R^n$ and its corresponding induced matrix norm of $\R^{n \times n}$.
  For each $v \in \Nodes$, we know by continuity of $q_v(x)$ that there exist $\beta_v > 0$ such that
  \[ q_v(x) \leq \beta_v\|x\|^{2d} \]
  for all $x \in \R^n$.
  Let $\beta = \max_{v \in \Nodes} \beta_v$.

  Therefore,
  \begin{align*}
    p_{v,k+1}(x)
    & \leq \frac{1}{\iterc^k}\sum_{s \in \Arcsin_k(v)} \beta_v \|A_s\|^{2d} \|x\|^{2d}\\
    & \leq \frac{\beta}{\iterc^k} \|x\|^{2d} \sum_{s \in \Arcsin_k(v)} \|A_s\|^{2d}.
  \end{align*}
  By Lemma~\ref{lem:rhoAG} and Remark~\ref{rem:path}, if
  $\iterc > \rho(A(G))\cdr$, $\lim_{k \to \infty} p_{v,k}(x) = 0$.
  We obtain the result by \lemref{conv}.
\end{proof}

\section{Finding high-growth sequences}
\label{sec:dualalgo}
In \secref{sos}, we introduced the SOS restriction of \progref{primalinf} with \progref{primal}.
In this section, we introduce \progref{dual}, the moment relaxation of \progref{dualinf}.
It turns out that \progref{primal} and \progref{dual} are dual to each other.
Indeed, the proof of \lemref{duality} can be translated verbatim in order to prove that \progref{dual} is the dual of \progref{primal}.
%Indeed, the proof of \theoref{duality} works as is with these programs so the Theorem holds for both the pair \progref{primalinf}/\progref{dualinf} and the pair \progref{primal}/\progref{dual}.

% \label{sec:dualalgo}
% In this section, we introduce \progref{dual}, the dual of \progref{primal},
% and use it to give an alternative proof of \theoref{second}.
% \theoref{second} %, \lemref{incp}
% and \eqref{eq:upperbound} give
% \begin{equation*}
%   \cdr \leq \cjsr \leq \jsrsos \leq \rho(A(\G))\cdr
% \end{equation*}
% which is equivalent to
% \begin{equation*}
%   \rho(A(\G))^{-1} \jsrsos \leq \cdr \leq \cjsr \leq \jsrsos.
% \end{equation*}
%
% This alternative proof shows that a solution of \progref{dual} with $\gamlb$
% provides a certificate for the lower bound $\rho(A(\G))^{-1} \gamlb$ to $\cdr$.
% This shows the inequality
% \[ \rho(A(\G))^{-1} \jsrsos \leq \cdr. \]
%
% %\todo[inline]{Also give a dual proof for the $\sqrt{n}$ bound. Maybe using John's theorem (or its generalization: Lewis' theorem \cite{tomczak1989banach}).}
%
% It turns out that we can extract high growth sequence from this certificate.
% We provide \algoref{prodl} parametrized by some integer parameter $l$
% that produces an high growth sequence using the solution of \progref{dual} with some $\gamlb < \jsrsos$.
% Let $g_{2d,l}$ be the asymptotic growth rate of the sequence produced.
% We show the following guarantee in \theoref{fourth}:
% \begin{equation*}
%   [\mdin_l]^{-1} \gamlb \leq g_{2d,l}
% \end{equation*}
% where $\mdin_l$ is defined in the \lon{}. % and $\gamlb$ is in the solution of \progref{dual}.

\subsection{Dual SOS program}
\label{sec:dual}
\begin{myprog}[Dual of \progref{primal}]
  \label{prog:dual}
  \begin{align}
    \notag
    \sup_{\pmu_{uv\sigma} \in \Rhd, \gamlb \in \R} \gamlb\\
    \label{eq:dual1}
    \sum_{(u,v,\sigma) \in \Arcs} \pushf{A_\sigma}{\pmu_{uv\sigma}} - \gamlb^{2d} \sum_{(v,w,\sigma) \in \Arcs} \pmu_{vw\sigma} & \in \Soshd, \forall v \in \Nodes,\\
    \label{eq:dual2}
    \pmu_{uv\sigma} & \in \Soshd, \quad \forall (u,v,\sigma) \in \Arcs,\\
    \label{eq:dual3}
    \sum_{(u,v,\sigma) \in \Arcs} \pmu_{uv\sigma}(\Sn) & = 1.
    % FIXME the normalization equation (12) should probably be  E...[ (sum_i xi^2)^d ]  rather than  (sum_i x_i^2d).
    %       In a certain sense, it does not matter (since the purpose is to make the problem affine rather than homogeneous).
    %       However, this form is more symmetric, since now things are orthogonally invariant.
    %       It affects the numerical Example but we should do it later for the journal version and/or final submission.
  \end{align}
\end{myprog}

%\todo{In case of irreducibility (as defined in \cite{philippe2016stability}), we can put $\succ$ instead of $\succeq$ in \eqref{eq:dualconsexp1}, should we add that ?}

% Note that the dual constraint \eqref{eq:dualconsexp1} is equivalent to:
% \begin{equation}
%   \label{eq:dualconsexp}
%   \sum_{(u,v,\sigma) \in \Arcs} \pE_{uv\sigma}[p(A_\sigma x)] \geq \gamma^{2d} \sum_{(v,w,\sigma) \in \Arcs} \pE_{vw\sigma}[p(x)]% \setminus \{0\}.
% \end{equation}
% for all $v \in \Nodes$ and for all $p(x) \in \Sigma_{\mathbf{2d}}$.

% \begin{myrem}
%   In the case of irreducibility, using Lemma~\ref{lem:strict}, one can derive a dual
%   where the inequality \eqref{eq:dualconsexp1} (or \eqref{eq:dualconsexp}) is strict. % FIXME strictly positive vs inte(\Sosh)
% \end{myrem}

It is important to note that a solution of \progref{dual} is not necessarily a solution of \progref{dualinf}.
First $\pmu_{uv\sigma}$ may not be a measure even if it belongs to $\Soshd$ as discussed in \secref{sosprog}.
Second, the left-hand side of \eqref{eq:dual1} may also not be a measure.
For this second concern, it helps to be more explicit.
Suppose for instance that we are in the quadratic case, i.e. $d = 1$.
In that case, if $\pmu \in \Soshd[2]$, there always exists a measure $\mu$ that has the moments of the pseudo-measure $\pmu$.
We can take for instance a Gaussian distribution with these second order moments.
Hence we can find Gaussian distributions $\mu_{uv\sigma}$ that have the second order moments $\pmu_{uv\sigma}$
and Gaussian distributions $\nu_v$ that have the second order moments given by the left-hand side of \eqref{eq:dual1}.
However, we may have
\[ \sum_{(u,v,\sigma) \in \Arcs} \pushf{A_\sigma}{\mu_{uv\sigma}} - \gamlb^{2d} \sum_{(v,w,\sigma) \in \Arcs} \mu_{vw\sigma} \neq \nu_v \]
as we only know that the left-hand side and right-hand side of the above equation have the same second order moments;
see \exemref{run3}.
%This comes from the fact that $\Soshd$ is included in the set of moments of mom
%Even in the quadratic case, i.e. $d=1$, for which \eqref{eq:dualconsexp2} is sufficient for $\pmu_{uv\sigma}$ to be a measure,

However, in some cases, we can recover a feasible solution of \progref{dualinf} from a feasible solution of \progref{dual}.
In these cases, by \cororef{instabcert}, this provides a lower bound to the CJSR.
%An important case is when the measures are atomic.
Moreover, there exist efficient techniques allowing to detect situations where the solution is moments of an atomic measure; see \cite{henrion2005detecting, laurent2009sums}.
%Moreover,
%we can efficiently detect when the measures are atomic and recover their atoms \cite{henrion2005detecting, laurent2009sums}.
Then, using the transformation of \propref{Bsc}, we can transform
these atomic measures into a feasible solution of a constrained scalar switched systems.
For such system, we could use the algorithm described in \lemref{instabcerts} but as pointed out in \cite{ahmadi2012joint},
computing the CJSR of a scalar system can easily be done by solving a maximum cycle mean problem for which efficient algorithm exists \cite{karp1978characterization}.

If we recover a feasible solution of \progref{dualinf} from a feasible solution of \progref{dual} with $\gamlb = \jsrsos$,
we can directly conclude that $\jsrsos = \cjsr$.
This is somewhat similar to the minimization of a multivariate polynomial using SOS where we can detect that we have reached the optimum
when the measure is atomic and recover the minimizers of the polynomial from the atoms of the measure. % mets une phrase moins angoissante que 'somewhat similar', et surtout mets une référence pour le lecteur intéressé (référence précise hein)

However, we may also check for atomic feasible solutions of \progref{dualinf} with $\gamlb < \jsrsos$ to provide lower bounds.
Moreover, in practice, $\jsrsos$ is computed by binary search on $\gamlb$ so we often have several such solutions.

\begin{myexem}
  \label{exem:simple2}
  Consider \exemref{simple1}.
  For $i = 1, 2, 3$, let $\pmu_i$ be the solution of \progref{dual} corresponding to the matrix $A_i$.
  For any $d$, we can see that the dual solution for $\gamma = 1$ is such that
  the only monomial $x^\alpha$ such that $\la\pmu_1, x^\alpha \ra$ (resp. $\la \pmu_2, x^\alpha \ra$, $\la \pmu_3, x^\alpha \ra$) is non-zero is $x_1^{2d}$ (resp. $x_2^{2d}$, $x_3^{2d}$)
  and $\la \pmu_1, x_1^{2d} \ra = \la \pmu_2, x_2^{2d} \ra = \la \pmu_3, x_3^{2d} \ra = 1/3$.
  Note that it means that $\pmu_1 = \delta_{(1,0,0)}/3$, $\pmu_2 = \delta_{(0,1,0)}/3$ and $\pmu_3 = \delta_{(0,0,1)}/3$
  where $\delta_x$ is the Dirac measure centered on $x$.
  %We remark that $\mu_1 = \delta_{(1,0,0)}/3$, $\mu_2 = \delta_{(0,1,0)}/3$ and $\mu_3 = \delta_{(0,0,1)}/3$ is solution of \progref{dualinf} with $\gamlb = 1$.
  Since these measures are solution to \progref{dualinf} with $\gamlb = 1$,
  by \cororef{instabcert}, this means that $\jsr \geq 1$.
\end{myexem}

\begin{myexem}
  \label{exem:run3}
  We continue the running example; see Example~\ref{exem:run1} and Example~\ref{exem:run2}.

  For all $d$, $\pmu_{212} = \pmu_{323} = \pmu_{344} = \pmu_{431} = 0$
  hence the node 4 is ``unused'' by the dual.
  For $2d = 2, 4, 6, 8$, $\pmu_{123} = \pmu_{231} = 0$ so
  the node $2$ is ``unused'' for low degree.

  At first, one could think that the dual variables can be used to reduce the systems,
  e.g. remove nodes or edges.
  However, as we will see, it would be a mistake to remove the node 2.

  For $2d = 10$, $\pmu_{123}$ and $\pmu_{231}$ are not zero and are of the ``same order or magnitude'' than $\pmu_{131}$ and $\pmu_{312}$.
  Then for $2d = 12$, $\pmu_{123}$ and $\pmu_{231}$ have ``larger magnitude'' than $\pmu_{131}$ and $\pmu_{312}$.
  This observation will be useful for Example~\ref{exem:run4}.

  We can see that while the shape of the primal variables changes a lot between $2d = 2$ and $2d = 4$ as mentioned in Example~\ref{exem:run2},
  the ``important'' change for the dual variables happens around $2d = 10$.

  It is also interesting to notice that the matrices corresponding to the dual variables have low rank.
  For example, for $2d = 2$, $\pmu_{131}$
  %(resp. $\pE_{312}$, $\pE_{331}$) is the Dirac measure $\delta_{(0.917,0.399)}$
  %(resp. $\delta_{(0.875,0.485)}$, $\delta_{(0.757,-0.653)}$).
  (resp. $\pmu_{312}$, $\pmu_{331}$) is the Dirac measure $0.324 \cdot \delta_{(0.917,0.399)}$
  (resp. $0.229 \cdot \delta_{(0.875,0.485)}$, $0.447 \cdot \delta_{(0.757,-0.653)}$).
  However, this is not a feasible solution of \progref{dualinf}.
  Indeed,
  while \eqref{eq:dualinf1} is satisfied for \node{} 1 since
  $\pushf{A_3}{\delta_{(0.875,0.485)}}$ gives $\delta_{(0.917,0.399)}$,
  $\pushf{A_1}{\delta_{(0.917,0.399)}}$ gives $\delta_{(1, -0.0273)}$ and 
  $\pushf{A_1}{\delta_{(0.757,-0.653)}}$ gives $\delta_{(0.423, -0.906)}$ so \eqref{eq:dualinf1} is not satisfied for \node{} 3.
\end{myexem}

\subsection{Generating high growth sequence}
\label{sec:algo}
In this section we give an algorithm that generates an infinite
sequence of matrices such that the asymptotic growth rate of the
product of the matrices is arbitrarily close to the CJSR.
Note that by \defref{cjsr}, this asymptotic growth rate must be
smaller than the CJSR.

%Suppose that we are given a polynomial $p_0(x) \in \inte(\Sosh)$,
%Since it is in the interior, $\pE_{uv\sigma}[p_0(x)] > 0$ for all $(u,v,\sigma) \in \Arcs$ such that $\pE_{uv\sigma}$ is not zero.
%Since the inequality \eqref{eq:dualconsexp} is strict, at least one pseudo-expectation $\pE_{v_1v_0\sigma_0}$ must be nonzero.

Given an \arc{} $e \in \Arcs$, let $\pE_e[p(x)] = \la \pmu_e, p(x) \ra$.
Given a polynomial $p_0(x) \in \inte(\Sosh)$ and an initial edge $(v_1,v_0,\sigma_1)$,
the algorithm builds a $\G^\Tr$-admissible
sequence $(v_1,v_0,\sigma_1),\allowbreak(v_2,v_1,\sigma_2), \ldots$ such that
\begin{equation}
	\label{eq:thetak}
  \theta_k \eqdef \pE_{v_{k}v_{k-1}\sigma_k}[p_0(A_{\sigma_1} \cdots A_{\sigma_k}x)]
\end{equation}
remains ``large'' for increasing $k$.
As we will see, using \lemref{thetatoA},
this implies that $A_{\sigma_1} \cdots A_{\sigma_k}$ has a ``large'' norm.

\begin{lemma}[{\cite[Lemma~6]{legat2016generating}}]
  \label{lem:sosbound}
  For any polynomial $p(x) \in \inte(\Sosh)$,
  there exists a constant $\beta$
  such that for any matrix $A$,
  \[ \beta\|A\|_2^{2d} p(x) - p(Ax) \quad \text{is SOS} \]
  where $\|A\|_2 = \rho(A^\Tr A)^{1/2}$ is the Euclidean norm.
  %Moreover we can choose $\beta = \kappa(Q)$ where $Q \in \mathcal{S}^{{n+d \choose d}}$
  %is the symmetric matrix such that $p(x) = (x^{[d]})^\Tr Qx^{[d]}$ for all $x \in \R^n$ % FIXME basis free approach ?
  %and $\kappa(Q) = \rho(Q) \rho(Q^{-1})$ is the condition number of $Q$.
% \begin{proof}
%   Consider the matrix $Q$ defined in the statement of the lemma.
%   Note that since $p(x) \in \inte(\Sosh)$, $Q$ is positive definite.
%   We can see that
%   \[ ((Ax)^{[d]})^\Tr Q(Ax)^{[d]} = (\ld{x})^\Tr (A^{[d]})^\Tr Q\ld{A}\ld{x}. \]
%   Moreover, for all $x \in \R^n$,
%   \[ (\ld{x})^\Tr Q\ld{x} \geq \|\ld{x}\|_2/\rho(Q^{-1}) \]
%   and by \lemref{svd},
%   \[ (\ld{x})^\Tr (A^{[d]})^\Tr Q\ld{A}\ld{x} \leq \rho(Q)\rho((\ld{A})^\Tr \ld{A})\|\ld{x}\|_2. \]
%   Using the fact that $\rho(A^{[d]}) = \rho(A)^d$ and $(AB)^{[d]} = A^{[d]}B^{[d]}$ for any matrix $A$ and $B$,
%   we have
%   \[ \kappa(Q) \rho(A^\Tr A)^{d} Q - (A^{[d]})^\Tr QA^{[d]} \succeq 0. \]
% \end{proof}
\end{lemma}

\begin{lemma}
  \label{lem:thetatoA}
  Let us consider a solution $(\, \pmu_e : e \in \Arcs \,)$ of \progref{dual}.
  For any polynomial $p(x) \in \inte(\Sosh)$,
  %and any pseudo-expectations $\pE_e \in \Soshd$ for each $e \in \Arcs$,
  there exists a positive constant $\tau$
  such that for any matrix $A \in \R^{n \times n}$ and \arc{} $e \in \Arcs$,
  \[ \pE_e[p(A x)] \leq \tau \|A\|_2^{2d} \]
  \begin{proof}
    If all pseudo-expectations are zero, the result is trivially true.
    Therefore we can suppose that at least one is nonzero.
    By \lemref{sosbound}, there exists a constant $\beta > 0$ such that
    \[ \beta \|A\|_2^{2d}p(x) - p(Ax) \text{ is SOS}. \]
    Hence
    \[
      \pE_e[p(Ax)] \leq \beta \|A\|_2^{2d} \pE_e[p(x)].
    \]
    We obtain the result with the constant
    \(
    \tau = \beta \max_{e \in \Arcs} \pE_{e}[p(x)].
    \)
    Since at least one pseudo-expectation is nonzero and $p(x)$ is in
    the interior of the SOS cone, $\tau > 0$.
  \end{proof}
\end{lemma}

%   The algorithm works similarly to the Markov Chain of the previous
%   section. In the previous section, at each step, we choose a incoming
%   path of length $l$ uniformly at random. In this section, we will try
%   to make smarter choice using the dual variables.
%   %To choose which incoming path of length $l$ to append to our sequence,
%   %we would like to compare them.
%   However the dual of the SOS cone only provides a partial order between
%   pseudo-expectations.  To extend this partial order to a full order, we
%   evaluate the pseudo-expectations at some fixed polynomial
%   $p_0(x) \in \inte(\Sosh)$.%, i.e. in the interior of the cone of SOS
%   %homogeneous polynomials of degree $2d$.
%
%   By \eqref{eq:dual3}, at least one dual variable is nonzero and
%   we can choose one of such \arcs{}, say $(v_0,v_{-1},\sigma_0)$.  We then
%   append paths of length $l$ one after the other to build a
%   $\G^\Tr$-admissible sequence.  To choose which path to add, we use
%   the partial order provided by the dual variables and $p_0(x)$ to give
%   the probability for each \arc{} to be taken.  This is described in
%   \algoref{prodl}.

\begin{algorithm}[!ht]
  \caption{Generates a sequence of large asymptotic growth using paths of length $l$.}
  \label{algo:prodl}
  \begin{algorithmic}
    \STATE{Given a feasible solution $(\, \pmu_e : e \in \Arcs \,)$ of \progref{dual}}
    \STATE{Pick an arbitrary polynomial $p_0(x) \in \inte(\Sosh)$}
    \STATE{Pick an \arc{} $(v_0,v_{-1},\sigma_0) \in \Arcs$ such that $\pE_{v_0v_{-1}\sigma_0}[p_0(x)] > 0$}
    \FOR{$k = 0,l,2l,\ldots$}
    \STATE Pick $s \in \argmax_{s \in \Arcsin_l(v_{k})} \pE_{s[1]}[p_k(A_s x)]$
    \STATE Set $(v_{k+l},\sigma_{k+l},\ldots,\sigma_{k+1},v_k) \leftarrow s$
    \STATE Set $p_{k+1} \leftarrow p_k(A_s x)$
    \ENDFOR
  \end{algorithmic}
\end{algorithm}

%   Under some condition we still have the guarantee given by
%   Lemma~\ref{lem:maxdegpathr} and under further conditions, it is
%   satisfied not only expectation but with probability 1.
%   We define the random variable
%   \[
%     \Theta_k \eqdef \pE_{S_k[1]}[p_0(A_{S_k}x)].
%   \]
%   Note that $\Theta_k = \pE_k[p_0(x)]$ where $\pE_k$ is the random
%   variable defined \eqref{eq:pEk}.

%We now show how to use the dual variables to provide a lower bound
%on the $2d$ moment of the random variable $\|A_{S_k}\|^{2d}$,
%that is, $\cdr^{2d}$.
%The dual constraint \eqref{eq:dual1} implies the following Lemma.

\lemref{maxdegpathl} provides a guarantee on the growth rate of $\theta_k$, defined in \eqref{eq:thetak},
using the dual constraint~\eqref{eq:dual1}.

\begin{lemma}
  \label{lem:dualconsl}
  Given \defAGe{},
  if $\pmu$ is a feasible solution of \progref{dual} then,
  for any edge $(u,v,\sigma) \in \Arcs$,
  the following holds:
  \begin{equation}
    \label{eq:dualconsl}
    \sum_{s \in \Arcsin_k(u)} \pushf{A_s}{\pmu_{s[1]}} \succeq \gamma^{2dk} \pmu_{uv\sigma}.
  \end{equation}
  \begin{proof}
    We prove \eqref{eq:dualconsl} by induction, the case of $k=0$ being trivial.
    We can rewrite the left-hand side of \eqref{eq:dualconsl} as
    \begin{equation}
      \label{eq:splitdualconsl}
      \sum_{s \in \Arcsin_k(u)} \pushf{\A_s}{\pmu_{s[1]}} = \sum_{s \in \Arcsin_{k-1}} \pushf{A_s}{\sum_{(u,s(1),\sigma) \in \Arcs} \pushf{A_\sigma}{\pmu_{us(1)\sigma}}}.
    \end{equation}
    By \eqref{eq:dual1},
    \[
      \sum_{(u,s(1),\sigma) \in \Arcs} \pushf{A_{\sigma}}{\pmu_{us(1)\sigma}}
      \succeq \gamma^{2d} \sum_{(s(1),w,\sigma) \in \Arcs} \pmu_{s(1)w\sigma}.
    \]
    Since the dual variables $\pmu_{s(1)w\sigma}$ of the right-hand side are in the dual of the SOS cone,
    and one of them is $\pmu_{s[1]}$, we have
    \[
      \sum_{(u,s(1),\sigma) \in \Arcs} \pushf{A_{\sigma}}{\pmu_{us(1)\sigma}}
      \succeq \gamma^{2d} \pmu_{s[1]}.
    \]
    Applying $\pushf{A_s}{}$ on both sides and using \eqref{eq:splitdualconsl} gives
    \[
      \sum_{s \in \Arcsin_k(u)} \pushf{A_s}{\pmu_{s[1]}} \succeq \gamma^{2d} \sum_{s \in \Arcsin_{k-1}(u)}\pushf{A_s}{\pmu_{s[1]}}.
    \]
  \end{proof}
\end{lemma}

\begin{lemma}
  \label{lem:maxdegpathl}
  Consider \defAGe{}.
  For any positive integers $d$ and $l$,
    using \progref{dual} with any $\gamma < \jsrsosa\\(\G,\A)$,
  \algoref{prodl} with paths of length $l$
  produces a $\G^\Tr$-admissible sequence
  $(v_1,v_0,\\\sigma_0),(v_2,v_1,\sigma_1), \ldots$ % of $\Arcs$
  for which the sequence of $\theta_k$ defined in \eqref{eq:thetak}
  satisfies the following inequality for all $k \geq 1$:
  \begin{equation*}
    \theta_{k}
    \geq \frac{\gamma^{2dl}}{\din(v_{k-l+1})}
    \theta_{k-l}
  \end{equation*}
  \begin{proof}
    By \lemref{dualconsl},
    \begin{align*}
      \sum_{s \in \Arcsin_k(v_{k-l+1})} \pE_{s[1]}[p_{k-l}(A_s)]
      & \geq \frac{\gamma^{2dl}}{\mdin_l} \theta_{k-l}.
    \end{align*}
    Since the value of $s$ chosen by \algoref{prodl} maximises $\pE_{s[1]}[p_{k-l}(A_s)]$,
    the left-hand side of the above inequality is smaller or equal to $\din(v_{k-l+1})\theta_k$.
  \end{proof}
\end{lemma}

\theoref{fourth} translates the guarantee on $\theta_k$ to a guarantee on $A_{\sigma_1} \cdots A_{\sigma_k}$ using \lemref{thetatoA}.

\begin{theorem}
  \label{theo:fourth}
  Consider \defAGe{}.
  For any positive integer $d$ and $l$,
  using \progref{dual} with any $\gamma < \jsrsos$,
  \algoref{prodl} with paths of length $l$
  produces a $\G^\Tr$-admissible sequence
  $(v_1,v_0,\sigma_0), (v_2,v_1,\sigma_1), \ldots$ % of $\Arcs$
  that satisfies the following inequality for all $k \geq 1$:
  \begin{equation*}
    \lim_{k \to \infty} \|A_{s_k}\|_2^\frac{1}{k} \geq \frac{\gamma}{\|(A(\G)^\Tr)^l\|_\infty^\frac{1}{2dl}}
  \end{equation*}
  where $s_k = (\sigma_k, \ldots, \sigma_1)$.
  \begin{proof}
    %Consider that $\gamma = \jsrsos - \epsilon$ for some $\epsilon > 0$.
    %Choose any strictly positive polynomial $p(x) \in \Sosh$.
    %Consider the sequence of \arcs{} whose existence is ensured by Lemma~\ref{lem:maxdegpath}.
    By \lemref{maxdegpathl}, for any $k$ multiple of $l$, % TODO lim sup ? lim exists ?
    \begin{equation*}
      %\label{eq:reck3}
      \pE_{s_k[1]}[p_0(A_{s_k}x)] \geq \frac{\gamma^{2dk}}{(\mdin_l)^\frac{k}{l}} \pE_{v_0v_{-1}\sigma_0}[p_0(x)]
    \end{equation*}

    By Lemma~\ref{lem:thetatoA},
    there exists a constant $\tau > 0$ such that
    \[
      \pE_{s_k[1]}[p_0(A_{s_k} x)] \leq \tau \|A_{s_k}\|^{2d}.
    \]

    Combining these two inequalities, we obtain
    \begin{align*}
      \tau \|A_{s_k}\|^{2d} & \geq \frac{\gamma^{2dk}}{(\mdin_l)^\frac{k}{l}} \pE_{v_0v_{-1}\sigma_0}[p_0(x)].
    \end{align*}
    Since $\pE_{v_0v_{-1}\sigma_0}$ is nonzero, $\pE_{v_0v_{-1}\sigma_0}[p_0(x)] > 0$.
    Therefore taking the $(2dk)$th root and the limit $k \to \infty$
    we obtain the result.
  \end{proof}
\end{theorem}

Taking the limit $l \to \infty$ and using \eqref{eq:rhoAG}, we see
that Theorem~\ref{theo:second} is a corollary of
Theorem~\ref{theo:fourth}.

% TODO build example for the conditions of this theorem

% \begin{myrem}
%   By Lemma~\ref{lem:sosbound},
%   in the proof of Theorem~\ref{theo:third}
%   we can take $\beta = \kappa(Q)$ where
%   $Q$ (resp. $P_0$) is the symmetric matrix such that $q(x) = (x^{[d]})^\Tr Qx^{[d]}$ (resp. $p_0(x) = (x^{[d]})^\Tr P_0x^{[d]}$) for all $x \in \R^n$.
%   We can see that $\rho(Q) \leq \rho(A_s)^{2d}$ and $\rho(Q^{-1}) = \check{\rho}(Q)^{-1} \leq [\check{\rho}(A_s)^{2d}\check{\rho}(P_0)]^{-1}$.
%   Also $\rho(A_s) \leq \cjsr^{k_1}$ and $\check{\rho}(A_s) \geq \check{\rho}(\G,\A)^{k_1}$.
%   So
%   \[ \kappa(Q) \leq \Big(\frac{\cjsr}{\check{\rho}(\G,\A)}\Big)^{k_1} \kappa(P_0). \]
% \end{myrem}

\begin{myexem}
  \label{exem:simple3}
  Suppose that we apply Algorithm~\ref{algo:prodl} with $l = 1$ to \exemref{simple2}
  %and the algorithm chooses $p_0(x) = \sum_\alpha c_\alpha x^\alpha$.
  and let us denote by $c_\alpha$ the coefficient of the monomial $x^\alpha$ in the polynomial $p_0(x)$ chosen arbitrarily by the algorihtm.
  The start of the sequence produced depends on the order between
  the coefficients $c_{(2d,0,0)},\allowbreak c_{(0,2d,0)},\allowbreak c_{(0,0,2d)}$.
  If $c_{(2d,0,0)}$ is the largest then
  the $\G$-admissible left-infinite sequence found is
  \[ \ldots,1,2,3,1,2,3,1,2,3. \]

  The product $A_{\sigma_1} A_{\sigma_2} A_{\sigma_3} \cdots = A_3A_2A_1A_3A_2A_1 \cdots $
  is periodic and has an asymptotic growth rate $\rho(A_{\sigma_1}A_{\sigma_2}A_{\sigma_3})^{1/3}\allowbreak = 1$.
  Hence $1 \leq \cjsr$.
  We saw in Example~\ref{exem:simple2} that $\jsrsos = 1$ for any $d$.
  Therefore $\cjsr = 1$.
  %The spectral radius of $A_{\sigma_1}A_{\sigma_2}A_{\sigma_3}$ is 1 so we know that $1 \leq \cjsr$
  %hence $\cjsr = 1$ since we also have $\cjsr \leq \jsrsos{2} = 1$.
  %The finiteness property holds.
\end{myexem}

\begin{figure}[!ht]
  \begin{subfigure}{0.48\textwidth}
  \centering
    \begin{tikzpicture}
      \SetGraphUnit{1.2}
      \GraphInit[vstyle=Dijkstra]
      \SetVertexNoLabel
      \SetUpEdge[style={->}]
      \Vertex[NoLabel=false,L=$v_k$]{0}
      \NOWE(0){4};
      \NOEA(0){5};
      \SO[NoLabel=false,L=$v_{k-1}$](0){11};
      \tikzset{VertexStyle/.append style = {dashed}}
      \SOWE(0){10};
      \SOEA(0){12};
      \Edge[label=?](4)(0);
      \Edge[label=?](5)(0);
      \Edge(0)(11);
      \tikzset{EdgeStyle/.append style = {dashed}}
      \Edge(0)(10);
      \Edge(0)(12);
    \end{tikzpicture}
    \caption{One iteration of Algorithm~\ref{algo:prodl} with $l = 1$.}
    \label{fig:path1}
  \end{subfigure}
  \begin{subfigure}{0.48\textwidth}
  \centering
    \begin{tikzpicture}
      \SetGraphUnit{1.2}
      \GraphInit[vstyle=Dijkstra]
      \SetVertexNoLabel
      \SetUpEdge[style={->}]
      %\SetVertexSimple[FillColor = blue]
      \Vertex[NoLabel=false,L=$v_k$]{0}
      \NOWE(0){4};
      \NOEA(0){5};
      \NOWE(4){6};
      \NO  (4){7};
      \NOEA(4){8};
      \NOEA(5){9};
      \SO[NoLabel=false,L=$v_{k-1}$](0){11};
      \tikzset{VertexStyle/.append style = {dashed}}
      \SOWE(4){1};
      \EA(0){2};
      \EA(2){3};
      \SOWE(0){10};
      \SOEA(0){12};
      \Edge[label=?](6)(4);
      \Edge[label=?](7)(4);
      \Edge[label=?](8)(4);
      \Edge[label=?](8)(5);
      \Edge[label=?](9)(5);
      \Edge[label=?](4)(0);
      \Edge[label=?](5)(0);
      \Edge(0)(11);
      \tikzset{EdgeStyle/.append style = {dashed}}
      \Edge(4)(1);
      \Edge(5)(2);
      \Edge(5)(3);
      \Edge(0)(10);
      \Edge(0)(12);
    \end{tikzpicture}
    \caption{One iteration of Algorithm~\ref{algo:prodl} with $l = 2$.}
    \label{fig:path2}
  \end{subfigure}
  \caption{Comparison between Algorithm~\ref{algo:prodl} with $l = 1$ and $l = 2$.
  The solid edge denotes the the edge $(v_k, v_{k-1}, \sigma_k)$
  and the edges with question marks denote the incoming path considered by the iteration of the algorithm.}
  \label{fig:paths}
\end{figure}
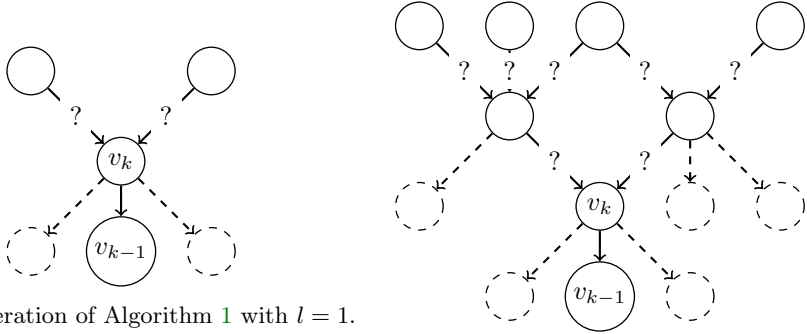

\subsection{Producing lower bounds}
\label{sec:lb}

%\todo[inline]{Use Elsner Lemma to give a guarantee on the period.}

By definition of the CJSR, the asymptotic growth rate of the norm of the product of any $\G$-admissible (or $\G^\Tr $-admissible) sequence of matrices gives a lower bound for the CJSR.
In particular the sequence produced by Algorithm~\ref{algo:prodl} provides a lower bound for the CJSR.

If there are two integers $\bar{k},k$ such that the sequence after $\bar{k}$ is periodic of period $k$,
the asymptotic growth rate of the norm is equal to the $k$th root of the spectral radius of the product of the matrices of one period.
This is due to the Gelfand's formula $\rho(A) = \lim_{k \to \infty} \|A^k\|^{1/k}$.
From the same identity,
we see that the spectral radius of the product of the matrices of one $\G$-admissible cycle gives a lower bound for the CJSR.

%If a sequence of matrices is periodic or if a sequence of matr
%If the sequence produced by Algorithm~\ref{algo:prod} or Algorithm~\ref{algo:prodl}
%is periodic, we can compute the asymptotic growth rate of the norm.
%From the identity ,
%the asymptotic growth rate of the norm of the product of the matrices of the sequence
%is equal to the spectral radius of the product of the matrices of one period.
%For this reason, the spectral radius of the product of the matrices of any cycle of the automaton
%gives a lower bound for the CJSR.

% More formally, for any natural $k$, we have
% \[ \rho_k(\G, \A) \leq \cjsr \]
% where
% \[ \rho_k(\G, \A) = \max \big\{\, \rho(A_c)^{1/k} : c \in \G_k, c \text{ is a cycle}\,\big\}. \]
% This is the generalization to the constrained case of the first inequality of the three members inequalities; see \cite[\Sect1.2.2.6]{jungers2009joint}.

To find lower bounds for the CJSR,
one could generate all the cycles of length smaller than some maximum length
and compute the spectral radius for all of them.
This brute force approach is not scalable because the number of paths considered grows exponentially with the maximum length.%
\footnote{The exponential growth of the brute force approach is the reason why one should choose a small $l$ for Algorithm~\ref{algo:prodl}.}
%The advantage of Algorithm~\ref{algo:prodl} over the brute force approach and Gripenberg's algorithm is that we can guarantee
%the maximum computing time that will be used to find a lower bound of a certain accuracy thanks to Theorem~\ref{theo:fourth}. % WRONG

Gripenberg~\cite{gripenberg1996computing} proposes a branch-and-bound algorithm that prunes the search using an a priori fixed absolute error.
Two other branch-and-bound variant exists:
the balanced complex polytope algorithm \cite{guglielmi2008algorithm}
and the invariant conitope algorithm~\cite{jungers2014lifted}.

These algorithms can also be used to produce a $\G$-admissible sequence of matrices of high asymptotic growth rate
by reproducing the cycles of high spectral radius infinitely.
The advantage of Algorithm~\ref{algo:prodl} is that it provides a guarantee of accuracy given in \theoref{fourth}.
Algorithm~\ref{algo:prodl} provides at the same time a high growth infinite trajectory and lower bounds of guaranteed accuracy.

We can compute lower bounds using the upper bound provided by Program~\ref{prog:primal} and Corollary~\ref{coro:mix}
but in practice the trajectories are periodic after some time $\bar{k}$ so we are able to compute
much better lower bounds than the pessimistic bound provided by Corollary~\ref{coro:mix}.
This is shown by the following example.

\begin{myexem}
  \label{exem:run4}
  We tried the atom extraction procedure and Algorithm~\ref{algo:prodl} for $l = 1$ and $l = 3$ on our running example;
  see Example~\ref{exem:run1}, Example~\ref{exem:run2} and Example~\ref{exem:run3}.
  The code used to obtain the results of this exemple can be found on the author's website. %\footnote{\url{https://perso.uclouvain.be/benoit.legat/}}.
  The result is shown in \figref{pedj}.
  We showed in \cite{legat2016generating} that the CJSR of the system is equal to 0.97482.
  We can see that this lower bound is found for $d = 4$ for $l=1$ and for $d=1$ for $l=3$.
  %\footnote{We can see that it is found for an earlier $d$ when $l=1$ compared to what is reported in \cite{legat2016generating}. The reason of this difference is explained in \remref{run}}.
  The atom extraction finds the lower bound 0.939255.
\end{myexem}

\begin{figure}[!ht]
  \begin{subfigure}{0.48\textwidth}
    \centering
    \includegraphics[width=\textwidth]{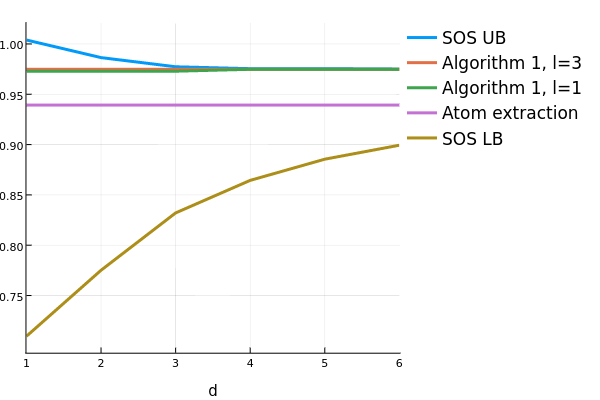}
  \end{subfigure}
  \begin{subfigure}{0.48\textwidth}
    \centering
    \includegraphics[width=\textwidth]{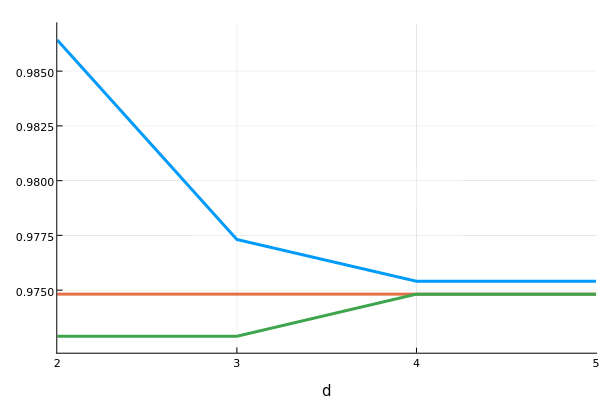}
  \end{subfigure}
  \caption{Result of \exemref{run4}. The SOS UB is the upper bound found by \progref{primal} and the SOS LB is obtained from this upper bound the guarantee given in \cororef{mix}. The value $d$ of horizontal axis corresponds to using polynomials of degree $2d$. The right figure is a zoom of the left figure.}
  \label{fig:pedj}
\end{figure}

\subsection{Improving the automaton-dependent bounds}
Summarizing our results above,
after solving Program~\ref{prog:primal}, we obtain an upper bound on
the CJSR and a lower bound thanks to Corollary~\ref{coro:mix}.
Running Algorithm~\ref{algo:prodl} provides lower bounds that we can
compute if the sequence produced is periodic and this lower bound will
always be at least as high than the lower bound produced by
Theorem~\ref{theo:second}.

In this subsection we show that there is another way to improve the lower
bound provided by Theorem~\ref{theo:second} using the dual solution.
The improved lower bound will never be higher than the bound provided by
Algorithm~\ref{algo:prodl} but it only requires checking which dual values are zero
so it almost does not require any computation.
This lower bound is provided by the \theoref{fifth} which is a
generalization of \cite[Theorem~3.10]{philippe2016stability}.

It is based on the fact that if in an optimal dual solution the dual
variable of an edge is 0 then removing the dual variable and the
primal constraint of this edge does not affect $\jsrsos$.
We could prove the following theorem using this fact but we give an
altervative proof using Algorithm~\ref{algo:prodl}.

\begin{definition}
  Consider \defAG{} and a positive integer $d$.
  The set of \arcs{} $\Arcs_{2d}$
  is the set of \arcs{} $e \in \Arcs$ such that there exists
  $\delta > 0$ such that for all $0 < \epsilon < \delta$, there is a
  solution of Program~\ref{prog:dual} with
  $\gamma \geq \jsrsos - \epsilon$ such that $\pmu_e = 0$.  We define
  the graph $\G_{2d}(\Nodes,\Arcs_{2d})$ using this set of \arcs{}.
\end{definition}

\begin{theorem}
  \label{theo:fifth}
  Consider \defAG{} and a positive integer $d$.
  The approximation given by Program~\ref{prog:primal} using homogeneous polynomials of degree $2d$
  satisfies:
  \[ \jsrsos \leq \rho(\G_{2d})^{\frac{1}{2d}}\cjsr \]
  where $A(\G_{2d})$ is the adjacency matrix of $\G_{2d}$.
  \begin{proof}
    For any edge $e \in \Arcs$ such that $\pmu_e = 0$, removing this
    edge does not violate any dual constraint \eqref{eq:dual1}.

    Using Theorem~\ref{theo:fourth} once we have removed all these
    edges and taking the limit $l \to \infty$ we obtain the result.
  \end{proof}
\end{theorem}

\begin{myrem}
  Note that it is not true that $\G_{2d_2} \subseteq \G_{2d_1}$ when
  $d_2 \geq d_1$.
  Indeed, as we have seen with Example~\ref{exem:run4} the dual
  variables of \arcs{} that were needed to have an s.m.p. are zero
  for $2d < 10$ so preventing Algorithm~\ref{algo:prodl} to choose
  these \arcs{} even for $l > 1$ prevents it to find the s.m.p. for
  $2d < 10$.
\end{myrem}

\section{Low rank reduction}
% TODO add a running example
\label{sec:lowrank}
Suppose we want to compute the CJSR of a finite set of matrices $\A \eqdef \{A_1, \ldots, A_m\} \subset \mathbb{R}^{n \times n}$ of rank at most $r$ constrained by an automaton $\G(V,E)$.
For $i = 1, \ldots, m$, since the matrix $A_i$ has rank at most $r$, there exists $X_i, Y_i \in \mathbb{R}^{n \times r}$ such that $A_i = X_iY_i^T$.
This can be used to build a new system with matrices of $\R^{r \times r}$ with the same CJSR.
This new system can therefore be used to reduce the computation the CJSR of system of low rank matrices to a system system matrices of small size.
Note that in the case $r=1$, it is known that the CJSR is computable in polynomial time~\cite{ahmadi2012joint}.

\begin{theorem}[Low Rank Reduction]
  \label{theo:lowrank}
  Consider a finite set of matrices $\A \eqdef \{A_1, \ldots, A_m\} \subset \mathbb{R}^{n \times n}$ of rank at most $r$
  constrained by an automaton $\G(V,E)$.

  For a fixed decomposition $A_\sigma = X_\sigma Y_\sigma^T$ for $\sigma = 1, \ldots, m$ where $X_\sigma, Y_\sigma \in \R^{n \times r}$,
  denote the set of matrices $\A' \eqdef \{A_{\sigma_1\sigma_2}' \mid \sigma_1, \sigma_2 = 1, \ldots, m\} \subset \mathbb{R}^{r \times r}$ where $A_{\sigma_1\sigma_2}' = Y_{\sigma_1}^TX_{\sigma_2}$.
  Define the graph $\G'(V',E')$ with $V' \eqdef E$ and
  %be the graph with one node in $V'$ for each edge in $E$ and
  %an edge $(i,j,ij) \in E'$ if $(i, j)$ is admissible in $\G$.
  \[
    E' \eqdef \{\, ((u,v,\sigma_1), (v,w,\sigma_2), \sigma_2\sigma_1) \mid (u,v,\sigma_1),(v,w,\sigma_2) \in E \,\}.
  \]
  %between a nodes representing $e_1 \in E$ and an edge representing $e_2 \in E$ in $\G'$ if it is possible to take the edge
  Then the two CJSR are the same:
  \[ \cjsr = \cjsrp. \]
  \begin{proof}
    As the CJSR does not depend on the norm used, we choose a norm $\|\cdot\|$ that is \emph{submultiplicative},
    that is $\|AB\| \leq \|A\| \|B\|$ for all matrices $A,B$.
    For example, any norm induced by a vector norm is submultiplicative.

    Let $\beta = \max_{i=1}^m\max\{\|X_i\|, \|Y_i^T\|\}$.
    If $\beta = 0$, then $\cjsr = 0 = \cjsrp$.
    Therefore we may assume that $\beta > 0$.
    Consider a positive integer $k$.
    We first show that $[\cjsrk]^k \leq \beta^2[\cjsrkp[k-1]]^{k-1}$.
    For any $\G$-admissible $(\sigma_1, \ldots, \sigma_k)$, we have
    \begin{equation}
      \label{eq:redu}
      A_{\sigma_k} \cdots A_{\sigma_2}A_{\sigma_1} = X_{\sigma_k}A_{\sigma_k\sigma_{k-1}}' \cdots A_{\sigma_3\sigma_2}' A_{\sigma_2\sigma_1}'Y_{\sigma_1}^T.
    \end{equation}
    using
    %\eqref{eq:redu} and
    the submultiplicativity of the norm chosen, we have
    \begin{align*}
      \|A_{\sigma_k} \cdots A_{\sigma_2}A_{\sigma_1}\|
      & \leq \|X_{\sigma_k}\|\|A_{\sigma_k\sigma_{k-1}}' \cdots A_{\sigma_3\sigma_2}' A_{\sigma_2\sigma_1}'\|\|Y_{\sigma_1}^T\|\\
      & \leq \beta^2\|A_{\sigma_k\sigma_{k-1}}' \cdots A_{\sigma_3\sigma_2}' A_{\sigma_2\sigma_1}'\|\\
      & \leq \beta^2 [\cjsrkp[k-1]]^{k-1}.
    \end{align*}

    The same way, we now show that $[\cjsrkp[k-1]]^{k-1} \leq \beta^2 [\cjsrk[k-2]]^{k-2}$.
    For any $\G'$-admissible $(\sigma_2\sigma_1, \ldots, \sigma_k\sigma_{k-1})$,
    \begin{align*}
      \|A_{\sigma_k\sigma_{k-1}}' \cdots A_{\sigma_3\sigma_2}' A_{\sigma_2\sigma_1}'\|
      & \leq \|Y_k^T\|\|A_{\sigma_{k-1}} \cdots A_{\sigma_2}\|\|X_1\|\\
      & \leq \beta^2 [\cjsrk[k-2]]^{k-2}.
    \end{align*}

    In summary, we have
    \[ \cjsrk \leq \beta^{\frac{2}{k}} [\cjsrkp[k-1]]^{\frac{k-1}{k}} \leq \beta^{\frac{4}{k}} [\cjsrk[k-2]]^{\frac{k-2}{k}}. \]
    Taking the limit $k \to \infty$ we get $\cjsr \leq \cjsrp \leq \cjsr$.
  \end{proof}
\end{theorem}

\begin{myexem}
  Consider an unconstrained switched system with 2 rank $r$ matrices $A_1, A_2$.
  This system is equivalent to the constrained switched system with automaton represented in \figref{redu1}.
  Its low rank reduction is represented in \figref{redu2}.
  \begin{figure}[!ht]
    \begin{subfigure}{0.48\textwidth}
      \centering
      \begin{tikzpicture}
        \SetGraphUnit{3}
        \GraphInit[vstyle=Dijkstra]
        \SetUpEdge[style={->}]
        \Vertex{0}
        \Loop[labelstyle={above=2pt,fill=white},dist=1.5cm,dir=NO,label={$A_1 = X_1Y_1^T$}](0)
        \Loop[labelstyle={below=2pt,fill=white},dist=1.5cm,dir=SO,label={$A_2 = X_2Y_2^T$}](0)
      \end{tikzpicture}
      \caption{Automaton $\G$.
      %To give an example for the notation, we
      We
      have $V = \{0\}$ and $E = \{(0,0,1),(0,0,2)\}$.}
      \label{fig:redu1}
    \end{subfigure}
    \begin{subfigure}{0.48\textwidth}
      \centering
      \begin{tikzpicture}
        \SetGraphUnit{3}
        %\GraphInit[vstyle=Dijkstra]
        \SetUpEdge[style={->}]%,
        %labelstyle = {sloped,draw}]
        \tikzset{EdgeStyle/.style = {distance = 30}}
        \Vertex[NoLabel=false]{1}
        \EA[NoLabel=false](1){2}
        \tikzset{EdgeStyle/.append style = {bend right=60}}
        \Loop[labelstyle={above=0pt},dist=1.5cm,dir=EA,label={$A_{11}' = Y_1^TX_1$}](1)
        \Loop[labelstyle={above=0pt},dist=1.5cm,dir=EA,label={$A_{22}' = Y_2^TX_2$}](2)
        \Edge[labelstyle={above=4pt, fill opacity=0, text opacity=1},label={$A_{12}' = Y_1^TX_2$}](1)(2)
        \Edge[labelstyle={below=4pt, fill opacity=0, text opacity=1},label={$A_{21}' = Y_2^TX_1$}](2)(1)
      \end{tikzpicture}
      \caption{Automaton $\G'$.
      We have $V' = \{1,2\}$ and
      $E' = \{(1,1,11),\allowbreak (1,2,12),\allowbreak (2,1,21),\allowbreak (2,2,22)\}$.}
      \label{fig:redu2}
    \end{subfigure}
    \caption{Simple example of the low rank reduction.}
    \label{fig:redu}
  \end{figure}
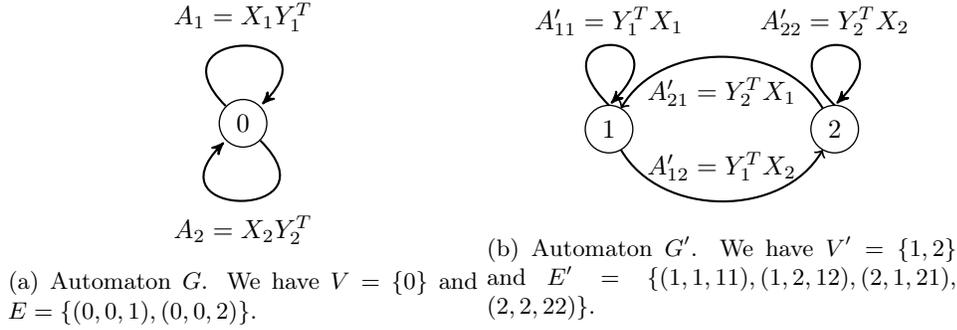
\end{myexem}

% Look at irreducible, write the algebraic condition, there should be finite time algo that solve it (maybe not efficient)
\begin{myrem}
  The matrices $X_\sigma, Y_\sigma$ of the factorization $A_\sigma = X_\sigma Y_\sigma^T$ are not unique.
  For any invertible matrix $S \in \mathbb{R}^{r \times r}$, $A_\sigma = (X_\sigma S)(S^{-1}Y_\sigma^T)$ also gives a factorization.
  However, if $\cjsrp$ is approximated using the sum of squares algorithm of \Sect\ref{sec:sos},
  any two factorizations will give the same approximation.
  The effect of using $X_\sigma S$ and $Y_\sigma S^{-T}$ instead of $X_\sigma$ and $Y_\sigma$ will simply be a linear change of variable of the polynomial $p_\sigma$;
  see \Sect\ref{sec:sos}.
  % TODO like rho(TAT^{-1}) -> irreducible, block, ...

  %If the rank of $A_\sigma$ is strictly less than $r$, we have the choice of choosing $X_\sigma$ and $Y_\sigma$ with the same rank of $A_\sigma$
  %or choosing $X_\sigma$ with the same rank of $A_\sigma$ and $Y_\sigma$ full rank for example.
  %For this choice, the question remains open. %\todo{TODO}
\end{myrem}

We now discuss the reduction quantitatively.
Suppose first that we want to approximate the unconstrained JSR of $m$ matrices of dimension $n$ and rank at most $r$.
Using Theorem~\ref{theo:lowrank}, we see that the unconstrained JSR is equal to the CJSR of $m^2$ matrices
of rank $r$ constrained by an automaton with $m$ nodes and $m^2$ edges,
the underlying directed graph of the automaton is the complete graph with a loop at every node.
This is summarized in Table~\ref{tab:jsrredu}.
%which requires $\bigoh(m^2r^2n)$ operations without taking the rank factorization into account. % TODO complexity of rank factorization

Suppose now that we want to approximate the CJSR of $m$ matrices of dimension $n$ and rank at most $r$ constrained by an automaton
of $|V|$ nodes and $|E|$ edges.
Using \theoref{lowrank}, we see that this CJSR is equal to the CJSR of matrices of dimension $r$.
The number of matrices in this CJSR might be lower than $m^2$ because
if $(\sigma_1, \sigma_2)$ is not $\G$-admissible, then we do not need $A_{\sigma_2\sigma_1}'$.
The size of each important quantity before and after the reduction is given by Table~\ref{tab:cjsrredu}.
We have just explained why $|\A'| \leq m^2$.
We can see that by definition of $E'$, $|E'| \leq |E|^2$ with equality if and only if $|V| = 1$, that is the unconstrained case.
Intuitively, the ``more'' the original CJSR is constrained, the ``sparser'' $\G'$ will be.

\begin{table}[!ht]
  \begin{subtable}{0.48\textwidth}
    \centering
    \begin{tabular}{l|cc}
      Quantity & $\A$ & $(\G',\A')$\\
      \hline
      dimension & $n$ & $r$\\
      rank & $r$ & $r$\\
      matrices & $m$ & $m$\\
      nodes &  & $m$\\
      edges &  & $m^2$\\
    \end{tabular}
    \caption{Unconstrained case.}%The quantity of the CJSR $(\G',\A')$ obtained by the reduction of Theorem~\ref{theo:lowrank} as a function of the quantity of the original JSR $\A$.}
    \label{tab:jsrredu}
  \end{subtable}
  \begin{subtable}{0.48\textwidth}
    \centering
    \begin{tabular}{l|cc}
      Quantity & $(\G,\A)$ & $(\G',\A')$\\
      \hline
      dimension & $n$ & $r$\\
      rank & $r$ & $r$\\
      matrices & $m$ & $\leq m^2$\\
      nodes & $|V|$ & $|E|$\\
      edges & $|E|$ & $\leq |E|^2$\\
    \end{tabular}
    \caption{Constrained case.}%The quantity of the CJSR $(\G',\A')$ obtained by the reduction of Theorem~\ref{theo:lowrank} as a function of the quantity of the original CJSR $(\G,\A)$.}
    \label{tab:cjsrredu}
  \end{subtable}
  \caption{Quantification of the low rank reduction of Theorem~\ref{theo:lowrank}. The quantity of the CJSR $(\G',\A')$ obtained by the reduction are expressed as a function of the quantity of the original CJSR $(\G,\A)$.}
  \label{tab:quant}
\end{table}

%\todo{Add example of low rank system (consensus ?)}
%\todo{We can remove linear spaces that are not in the invariant space containing the JSR}

\section{Conclusions}
We have analysed the dual of the SOS Lyapunov program for switched systems
and shown how to leverage it to study the system stability.
We also generalized the whole approach to the constrained switched systems,
a class of systems that has attracted increasing attention recently.

It turns out from our analysis that these two concepts are intrinsically related:
Our \theoref{fourth}, which leverages the dual of the classical JSR algorithm,
actually naturally applies to the constrained case;
Even more, \propref{Bsc} transforms an unconstrained system into a scalar constrained one for the purpose of computing a lower bound.
Finally, we show in \theoref{lowrank} that unconstrained systems with low rank matrices naturally lead to the definition of an auxiliary \emph{constrained} system.

%As it turns out from our analysis, these two concepts are intrinsically related.
%%The study of both topics in the same paper is motivated by the following results.
%We gave a rounding algorithm (\algoref{prodl}) that generates from a solution of the dual program a
%switching sequence of growth rate guaranteed by \theoref{fourth}. From this guarantee, we deduced
%a new estimate of the accuracy of the SOS-based approximation algorithm for the CJSR
%which is better than the previously existing one for sufficiently large SOS degree.
%We also showed in \propref{Bsc} that given atomic measures solution of this dual program,
%we can extract a high growth-rate trajectory of the switched systems
%by transforming them to the solution a scalar constraint switched systems.
%Finally, in \theoref{lowrank}, we introduced a reduction of the estimation of the (C)JSR of switched systems
%of low rank matrices to the CJSR of a switched system of low dimension matrices.

We have introduced two techniques to generate lower bounds from the solution of the SOS dual program.
In practice, these techniques provide periodic trajectories of high asymptotic growth rate.
Since the SOS program can be solved efficiently,
does this give an efficient algorithm to generate lower bounds on the CJSR with guaranteed accuracy?
This is not clear, because our algorithm provides firm guarantees only when the computed measures are atomic, which is not always the case.

More generally, the techniques developed in this work, based on
generating ``bad'' trajectories for a dynamical system via dual solutions% of the natural convex problems used for analysis
, naturally extend to many other problems in systems theory. We are currently
exploring such possibilities.

\bibliographystyle{siamplain}
\bibliography{biblio}

\begin{thebibliography}{10}

\bibitem{ahmadi2012joint}
{\sc A.~A. Ahmadi and P.~A. Parrilo}, {\em {J}oint spectral radius of rank one
  matrices and the maximum cycle mean problem.}, in CDC, 2012, pp.~731--733.

\bibitem{ahuja1993network}
{\sc R.~K. Ahuja, T.~L. Magnanti, and J.~B. Orlin}, {\em {N}etwork {F}lows:
  {T}heory, {A}lgorithms, and {A}pplications}, Prentice-Hall, Inc., Upper
  Saddle River, NJ, USA, 1993.

\bibitem{barak2012hypercontractivity}
{\sc B.~Barak, F.~G. Brandao, A.~W. Harrow, J.~Kelner, D.~Steurer, and
  Y.~Zhou}, {\em {H}ypercontractivity, sum-of-squares proofs, and their
  applications}, in Proceedings of the forty-fourth annual ACM Symposium on
  Theory of Computing, ACM, 2012, pp.~307--326.

\bibitem{berger1992bounded}
{\sc M.~A. Berger and Y.~Wang}, {\em Bounded semigroups of matrices}, Linear
  Algebra and its Applications, 166 (1992), pp.~21--27.

\bibitem{blondel2005computationally}
{\sc V.~D. Blondel and Y.~Nesterov}, {\em {C}omputationally efficient
  approximations of the joint spectral radius}, SIAM Journal on Matrix Analysis
  and Applications, 27 (2005), pp.~256--272.

\bibitem{blondel2000boundedness}
{\sc V.~D. Blondel and J.~N. Tsitsiklis}, {\em {T}he boundedness of all
  products of a pair of matrices is undecidable}, Systems \& Control Letters,
  41 (2000), pp.~135--140.

\bibitem{choi1995sums}
{\sc M.-D. Choi, T.~Y. Lam, and B.~Reznick}, {\em {S}ums of squares of real
  polynomials}, in Proceedings of Symposia in Pure mathematics, vol.~58,
  American Mathematical Society, 1995, pp.~103--126.

\bibitem{claeys2016modal}
{\sc M.~Claeys, J.~Daafouz, and D.~Henrion}, {\em Modal occupation measures and
  lmi relaxations for nonlinear switched systems control}, Automatica, 64
  (2016), pp.~143--154.

\bibitem{dai2012gel}
{\sc X.~Dai}, {\em {A} {G}el'fand-type spectral radius formula and stability of
  linear constrained switching systems}, Linear Algebra and its Applications,
  436 (2012), pp.~1099--1113.

\bibitem{elsner1995generalized}
{\sc L.~Elsner}, {\em {T}he generalized spectral-radius theorem: an
  analytic-geometric proof}, Linear Algebra and its Applications, 220 (1995),
  pp.~151--159.

\bibitem{fekete1923verteilung}
{\sc M.~Fekete}, {\em {\"U}ber die {V}erteilung der {W}urzeln bei gewissen
  algebraischen {G}leichungen mit ganzzahligen {K}oeffizienten}, Mathematische
  Zeitschrift, 17 (1923), pp.~228--249.

\bibitem{gripenberg1996computing}
{\sc G.~Gripenberg}, {\em {C}omputing the joint spectral radius}, Linear
  Algebra and its Applications, 234 (1996), pp.~43--60.

\bibitem{guglielmi2008algorithm}
{\sc N.~Guglielmi and M.~Zennaro}, {\em {A}n algorithm for finding extremal
  polytope norms of matrix families}, Linear Algebra and its Applications, 428
  (2008), pp.~2265--2282.

\bibitem{hardy1952inequalities}
{\sc G.~H. Hardy, J.~E. Littlewood, and G.~P{\'o}lya}, {\em {I}nequalities},
  Cambridge university press, 1952.

\bibitem{hare2011explicit}
{\sc K.~G. Hare, I.~D. Morris, N.~Sidorov, and J.~Theys}, {\em {A}n explicit
  counterexample to the {L}agarias–{W}ang finiteness conjecture}, Advances in
  Mathematics, 226 (2011), pp.~4667 -- 4701,
  \url{https://doi.org/http://dx.doi.org/10.1016/j.aim.2010.12.012},
  \url{http://www.sciencedirect.com/science/article/pii/S0001870810004457}.

\bibitem{henrion2005detecting}
{\sc D.~Henrion and J.-B. Lasserre}, {\em Detecting global optimality and
  extracting solutions in gloptipoly}, in Positive polynomials in control,
  Springer, 2005, pp.~293--310.

\bibitem{jungers2009joint}
{\sc R.~Jungers}, {\em {T}he joint spectral radius: theory and applications},
  vol.~385, Springer Science \& Business Media, 2009.

\bibitem{jungers2014lifted}
{\sc R.~M. Jungers, A.~Cicone, and N.~Guglielmi}, {\em {L}ifted polytope
  methods for computing the joint spectral radius}, SIAM Journal on Matrix
  Analysis and Applications, 35 (2014), pp.~391--410.

\bibitem{jungers2011fast}
{\sc R.~M. Jungers and V.~Y. Protasov}, {\em {F}ast methods for computing the
  p-radius of matrices}, SIAM Journal on Scientific Computing, 33 (2011),
  pp.~1246--1266.

\bibitem{karp1978characterization}
{\sc R.~M. Karp}, {\em {A} characterization of the minimum cycle mean in a
  digraph}, Discrete Mathematics, 23 (1978), pp.~309 -- 311,
  \url{https://doi.org/http://dx.doi.org/10.1016/0012-365X(78)90011-0},
  \url{http://www.sciencedirect.com/science/article/pii/0012365X78900110}.

\bibitem{kozyakin2014berger}
{\sc V.~Kozyakin}, {\em {T}he {B}erger--{W}ang formula for the {M}arkovian
  joint spectral radius}, Linear Algebra and its Applications, 448 (2014),
  pp.~315--328.

\bibitem{lasserre2009moments}
{\sc J.~B. Lasserre}, {\em Moments, positive polynomials and their
  applications}, World Scientific, 2009.

\bibitem{laurent2009sums}
{\sc M.~Laurent}, {\em {S}ums of squares, moment matrices and optimization over
  polynomials}, in Emerging applications of algebraic geometry, Springer, 2009,
  pp.~157--270.

\bibitem{legat2016generating}
{\sc B.~Legat, R.~M. Jungers, and P.~A. Parrilo}, {\em {G}enerating unstable
  trajectories for {S}witched {S}ystems via {D}ual {S}um-{O}f-{S}quares
  techniques}, in Proceedings of the 19th International Conference on Hybrid
  Systems: Computation and Control, HSCC '16, ACM, 2016, pp.~51--60,
  \url{https://doi.org/10.1145/2883817.2883821},
  \url{http://doi.acm.org/10.1145/2883817.2883821}.

\bibitem{nesterov2000squared}
{\sc Y.~Nesterov}, {\em {S}quared functional systems and optimization
  problems}, in High performance optimization, Springer, 2000, pp.~405--440.

\bibitem{ogura2016efficient}
{\sc M.~Ogura, V.~M. Preciado, and R.~M. Jungers}, {\em Efficient method for
  computing lower bounds on the p-radius of switched linear systems}, Systems
  \& Control Letters, 94 (2016), pp.~159--164.

\bibitem{parrilo2000structured}
{\sc P.~A. Parrilo}, {\em {S}tructured semidefinite programs and semialgebraic
  geometry methods in robustness and optimization}, PhD thesis, Citeseer, 2000.

\bibitem{parrilo2008approximation}
{\sc P.~A. Parrilo and A.~Jadbabaie}, {\em {A}pproximation of the joint
  spectral radius using sum of squares}, Linear Algebra and its Applications,
  428 (2008), pp.~2385--2402.

\bibitem{parrilo2003semidefinite}
{\sc P.~A. Parrilo and S.~Lall}, {\em Semidefinite programming relaxations and
  algebraic optimization in control}, European Journal of Control, 9 (2003),
  pp.~307--321.

\bibitem{philippe2016stability}
{\sc M.~Philippe, R.~Essick, G.~E. Dullerud, and R.~M. Jungers}, {\em Stability
  of discrete-time switching systems with constrained switching sequences},
  Automatica, 72 (2016), pp.~242--250.

\bibitem{protasov1997generalized}
{\sc V.~Y. Protasov}, {\em {T}he generalized joint spectral radius. {A}
  geometric approach}, Izvestiya: Mathematics, 61 (1997), p.~995,
  \url{http://stacks.iop.org/1064-5632/61/i=5/a=A05}.

\bibitem{reznick1978extremal}
{\sc B.~Reznick}, {\em {E}xtremal {PSD} forms with few terms}, Duke Math. J.,
  45 (1978), pp.~363--374, \url{https://doi.org/10.1215/S0012-7094-78-04519-2},
  \url{http://dx.doi.org/10.1215/S0012-7094-78-04519-2}.

\bibitem{rota1960note}
{\sc G.-C. Rota and W.~Strang}, {\em {A} note on the joint spectral radius},
  Proceedings of the Netherlands Academy,  (1960).
\newblock 22:379--381.

\bibitem{shor1987class}
{\sc N.~Shor}, {\em {C}lass of global minimum bounds of polynomial functions},
  Cybernetics and Systems Analysis, 23 (1987), pp.~731--734.

\bibitem{wang2014stability}
{\sc Y.~Wang, N.~Roohi, G.~E. Dullerud, and M.~Viswanathan}, {\em {S}tability
  of linear autonomous systems under regular switching sequences}, in Decision
  and Control (CDC), 2014 IEEE 53rd Annual Conference on, IEEE, 2014,
  pp.~5445--5450.

\bibitem{zhou2002p}
{\sc D.-X. Zhou}, {\em {T}he p-norm joint spectral radius and its applications
  in wavelet analysis}, AMS IP Studies in Advanced Mathematics, 25 (2002),
  pp.~305--326.

\end{thebibliography}

\appendix
\section{Stability certificates and duality}
\begin{theorem}
  \label{theo:stabcert}
  Consider \defAGe{}.
% If there exists a strictly positive homogeneous function $f_v(x)$ for all $v \in \Nodes$
% such that
% \[ f_{v}(A_\sigma x) \leq \gamub f_{u}(x) \]
% holds for all \arc{} $(u,v,\sigma) \in \Arcs$.
% Then $\lim_{k \to \infty} \cjsrk \leq \gamub$.
  We have
  \[ \lim_{k \to \infty} \cjsrk \leq \gamubopt. \]
  \begin{proof}
    Consider a norm $\|\cdot\|$ of $\R^n$ and its corresponding induced matrix norm of $\R^{n \times n}$.
    For each $v \in \Nodes$, we know by compactness of the unit ball in $\R^n$, continuity and strict positivity of $f_v(x)$ that there exist $0 < \alpha_v \leq \beta_v$ such that
    \[ \alpha_v \|x\| \leq f_v(x) \leq \beta_v\|x\| \]
    for all $x \in \R^n$.
    Let $\alpha = \min_{v \in \Nodes} \alpha_v$ and $\beta = \max_{v \in \Nodes} \beta_v$.

    For a $\G$-admissible $k$-uple $(\sigma_1, \sigma_2, \ldots, \sigma_k)$,
    \[ \|A_{\sigma_k} \cdots A_{\sigma_1}\| = \sup_{x \neq 0} \frac{\|A_{\sigma_k} \cdots A_{\sigma_1}x\|}{\|x\|}. \]
    Consider a path such that the $i$th edge has label $\sigma_i$ for $i = 1, \ldots, k$
    and denote the intermediary nodes of that path as $v_0, v_1, \ldots, v_k$.
    For any $x \in \R^n$, we have
    \begin{align*}
      \|A_{\sigma_k} \cdots A_{\sigma_1}x\| %&
                                            \leq \alpha_{v_k} f_{v_k}(A_{\sigma_k} \cdots A_{\sigma_1} x)%\\
                                            %&
                                            \leq \alpha_{v_k} \gamub f_{v_{k-1}}(A_{\sigma_{k-1}} \cdots A_{\sigma_1} x)%\\
                                            %&
                                            \leq \alpha_{v_k} \gamub^{k} f_{v_0}(x)
    \end{align*}
    and
    \[ \|x\| \geq \beta_{v_0} p_{v_0}(x) \]
    hence
    \[ \|A_{\sigma_k} \cdots A_{\sigma_1}\| \leq \frac{\beta_{v_0}}{\alpha_{v_k}} \gamub^k \leq \frac{\beta}{\alpha} \gamub^k. \]
    Taking the $k$th root, the limit $k \to \infty$ and using \defref{cjsr} we obtain the result.
  \end{proof}
\end{theorem}

\begin{lemma}[No duality gap]
  \label{lem:duality}
  For a fixed $\gamma$,
  \begin{description}
    \item[Weak duality] If \progref{primalinf} (resp. \progref{dualinf}) is feasible for $\gamub = \gamma$ (resp. $\gamlb = \gamma$) then \progref{dualinf} (resp. \progref{primalinf}) is infeasible for all $\gamlb < \gamma$ (resp. $\gamub > \gamma$).
    \item[Strong duality] If \progref{primalinf} (resp. dual) is infeasible for $\gamub = \gamma$ (resp. $\gamlb = \gamma$) then \progref{dualinf} (resp. \progref{primalinf}) is feasible for $\gamlb = \gamma$ (resp. $\gamub = \gamma$).
  \end{description}

  In other words, there exists a value $\gamma^*$ such that
  for every $\gamma > \gamma^*$, there exists a feasible solution to \progref{primalinf} \progref{primalinf} and
  for every $\gamma < \gamma^*$, there exists a feasible solution to \progref{dualinf} \progref{dualinf}.
  Moreover, either \progref{primalinf} program, \progref{dualinf} program or both have a feasible solution with $\gamma = \gamma^*$.
  \begin{proof}
    Consider the hyperplane
    \[ C \eqdef \Big\{\, (f_v : v \in \Nodes) \in \F^{|\Nodes|} \, \Big| \sum_{v \in V} \int_{\Sn} f_v(x) \dif x = 1 \,\Big\} \]
    and the map
    %\[ C \colon \F^{|\Nodes|} \to \F^{|\Nodes|} : S \mapsto \Big\{\, (f_v : v \in \Nodes) \in S \, \Big| \sum_{v \in V} \int_{\Sn} f_v(x) \dif x = 1 \,\Big\}. \]
    \[ \mapc_\gamma \colon \F^{|\Nodes|} \to \F^{|\Arcs|} : (f_v : v \in \Nodes) \mapsto (\gamma f_u(x) - f_v(A_\sigma x) : (u,v,\sigma) \in \Arcs). \]

    Given a fixed $\gamma$,
    \progref{primalinf} has no solution for $\gamub = \gamma$
    if and only if $\mapc_{\gamma}(\Fpp^{|\Nodes|} \cap C) \cap \Fp^{|E|} = \emptyset$.
    Since $\Fpp^{|\Nodes|} \cap C$ is compact, so is $\mapc_{\gamma}(\Fpp^{|\Nodes|} \cap C))$.
    We know that a compact set and a closed set have no intersection if and only if there exist a strict separating hyperplane separating the two sets.
    That is, a measure $\mu \in \Fbd$ such that $\la \mu, f \ra \geq 0$ for all $f \in \Fp^{|E|}$ and $\la \mu, f \ra < 0$ for all $f \in \mapc_{\gamma}(\Fpp^{|\Nodes|} \cap C)$.
    The first condition is simply $\mu \in \Fbpd$.
    For the second condition, we remark that
    \[ \mapc_{\gamma}(\Fpp^{|\Nodes|} \cap C) = \mapc_{\gamma}(\inte(\Fp^{|\Nodes|}) \cap C) = \relint\mapc_{\gamma}(\Fp^{|\Nodes|} \cap C). \]
    We have $\la \mu, f \ra < 0$ for all $f \in \relint\mapc_{\gamma}(\Fp^{|\Nodes|} \cap C)$ if and only if
    $\la \mu, f \ra \leq 0$ for all $f \in \mapc_{\gamma}(\Fp^{|\Nodes|} \cap C)$ and
    \begin{equation}
      \label{eq:affinhyp}
      \exists f \in \mapc_{\gamma}(\Fp^{|\Nodes|} \cap C) : \la \mu, f \ra \neq 0.
    \end{equation}

    Therefore, if \progref{primalinf} has no solution for $\gamub = \gamma$ then there exists
    a \emph{nonzero} measure $\mu \in (\Fbpd)^{|E|}$ such that for all $f \in C$ and $(u,v,\sigma) \in E$,
    \begin{equation}
      \label{eq:dualconsfirst}
      \sum_{v \in V} \sum_{(v,u,\sigma) \in E} \bar{\gamma} \Exp_{vu\sigma}[f_v(x)] \leq \sum_{v \in V} \sum_{(u,v,\sigma) \in E} \Exp_{uv\sigma}[f_v(A_\sigma x))]
    \end{equation}
    and \eqref{eq:affinhyp} holds.

    Note that if the inequality \eqref{eq:dualconsfirst} is respected for some $f \in C$, it is also respected for $\lambda f$ for all $\lambda > 0$.
    So we can impose that the inequality should be respected for all $f \in \Fp^{|V|} \setminus \{0\}$.

    The constraint~\eqref{eq:dualconsfirst} must be true for all $f \in \Fp^{|V|} \setminus \{0\}$ so in particular in the case where there is a node $v \in V$ such that $f_u(x) = 0$ for all $u \neq v$.
    Therefore we must have
    \[ \gamma \sum_{(v,u,\sigma) \in E} \Exp_{vu\sigma}[f_v(x)] \leq \sum_{(u,v,\sigma) \in E} \Exp_{uv\sigma}[f_v(A_\sigma x))], \quad \forall f_v \in \Fp \]
    for all $v \in V$.
    This is \eqref{eq:dualinf1} so the strong duality is proven.

    To show the weak duality, we show that if there exists a dual solution $\mu$ for $\gamlb=\gamma$ then \eqref{eq:dualinf1} and \eqref{eq:affinhyp} are satisfied for all $\gamlb < \gamma$.
    We know that \eqref{eq:dualinf1} is satisfied for $\gamma$ so the constraint \eqref{eq:dualinf1} is also satisfied for any $\gamlb < \gamma$.
    %For \eqref{eq:affinhyp}, we consider \eqref{eq:dualinf1} and \eqref{eq:affinhyp} implies that with $f(x) = \|x\|$.
    Using \eqref{eq:dualconsfirst} and \eqref{eq:dualinf3} with $f_v(x) = \|x\|$ for all $v \in \Nodes$,
    we have $\la \mu, f \ra < 0$ for all $\gamlb < \gamma$.

    %Therefore we can define the notation $\ell^{A^*}$ such that $\ell^{A^*}(p(x)) = \ell(p(Ax))$ for all polynomial $p \in \R[x]_{\mathbf{2d}}$.
    % We can thus rewrite the inequality as
    % \[
    %   \sum_{(u,v,\sigma) \in E} \mathcal{A}_\sigma^*(\ell_{uv\sigma}) - \gamma^{2d} \sum_{(v,u,\sigma) \in E} \ell_{vu\sigma} \in \inte(\Soshs)
    % \]
    % for all $v \in V$.
    %
    % The constraint is thus rewritten as
    % \[
    %   \sum_{(u,v,\sigma) \in E} \mathcal{A}_\sigma^*(\pE_{uv\sigma}) - \gamma^{2d} \sum_{(v,u,\sigma) \in E} \pE_{vu\sigma} \in \inte(\Sigma_{\mathbf{2d}}^*)
    % \]
    % or alternatively, for node $v \in V$ and any polynomial $p \in \Sigma_{\mathbf{2d}} \setminus \{0\}$,
    % \begin{equation}
    %   \label{eq:dualinf}
    %   \sum_{(u,v,\sigma) \in E} \pE_{uv\sigma}[p(A_\sigma x)] > \gamma^{2d} \sum_{(v,u,\sigma) \in E} \pE_{uv\sigma}[p(x)].
    % \end{equation}
  \end{proof}
\end{lemma}

\section{The $p$-radius}
\label{app:pradius}
We extend the definition of the $p$-radius to the constrained case.
%It has many applications and it can be computed exactly for even $p$.
%The $p$-radius is a generalization of the joint spectral radius.
\begin{definition}[Constrained $p$-radius]
  The \emph{constrained $p$-radius} of
  \defAGe{}, denoted as $\cpr$,
  is
  \begin{equation*}
    \cpr =
    %\begin{cases}
    \lim_{k \to \infty} \left[|\Arcs_k|^{-1} \sum_{v \in V} \cprkv{k}{v}\right]^{\frac{1}{pk}}
%, & p < \infty\\
        %\cjsr, & p = \infty.
    %\end{cases}
  \end{equation*}
  where
  \begin{align*}
    \cprkv{k}{v} & =
    \sum_{s \in \Arcsout_k(v)} \|A_s\|^p.
  \end{align*}
  Thus, the CJSR can be defined as the constrained $p$-radius for $p = \infty$.
\end{definition}

%\begin{align*}
%  \cprkv{k}{v} & = \max_{v \in V}.
%\end{align*}

\begin{myrem}
  \label{rem:path}
  Since $\G$ is assumed to be strongly connected,
  we could give the following equivalent definition
  \begin{equation}
    \label{eq:cprk2}
    \cpr = \lim_{k \to \infty}
    \left[\max_{v \in \Nodes} [\dout_k(v)]^{-1} \cprkv{k}{v}\right]^{\frac{1}{pk}}
  \end{equation}
% where
% \[
%   \cprkv{k}{v} = \max_{v \in V} [\cprkv{k}{v}]^{\frac{1}{pk}}.
%   \left[
%     [\dout_k(v)]^{-1} \sum_{\substack{s \in \Arcs_k\\s(1)=v}} \|A_s\|^p
%   \right]^{\frac{1}{pk}}.
% \]
  or the same definition with ``$\din_k(v)$'' instead of ``$\dout_k(v)$'' and ``$s \in \Arcsin_k(v)$'' instead of ``$s \in \Arcsout_k(v)$'' in the definition of $\cprkv{k}{v}$.

  %Indeed, let $\alpha_v = [\dout_k(v)]^{-1} \cprkv{k}{v}]^{\frac{1}{pk}}$.
  %Indeed, there is at most a factor $|\Nodes|^{\frac{1}{k}}$ between
  %the two definitions but this factor tends to $1$ as $k \to \infty$.
\end{myrem}

By the equivalence of norms, the definition of the $p$-radius does not depend on the norm used.

We can show that the $p$-radius is well defined using the following classical result, known as \emph{Fekete's Lemma} \cite{fekete1923verteilung}.
\begin{lemma}
  \label{lem:fekete}
  Let $\{a_n\} : n \geq 1$ be a sequence of real numbers such that
  \[ a_{m + n} \leq a_m + a_n. \]
  Then the limit
  \[ \lim_{n \to \infty} \frac{a_n}{n} \]
  exists and is equal to $\inf\left\{\frac{a_n}{n}\right\}$.
\end{lemma}
\begin{lemma}
  \label{lem:feketemax}
  Consider \defAGe{} and
  the sequence $(a_k)_k = \max_{v \in V} \cprkv{k}{v}$ with a submultiplicative norm.
  The sequence $\sqrt[k]{a_k}$ converges when $k \to \infty$.
  Moreover,
  \[ \lim_{k \to \infty} \sqrt[k]{a_k} = \inf \{\sqrt[k]{a_k}\}. \]

  \begin{proof}
    By submultiplicativity, for any $v \in V$, $k$ and any $k_1,k_2 \geq 0$ such that $k_1 + k_2 = k$,
    \begin{align*}
      \cprkv{k}{v}
      & = \sum_{u \in V} \sum_{s_1 \in \Arcs_{k_1}(v,u),s_2\in \Arcsout_{k_2}(u)} \|A_{s_2}A_{s_1}\|^p\\
      & \leq \sum_{u \in V} \sum_{s_1 \in \Arcs_{k_1}(v,u),s_2\in \Arcsout_{k_2}(u)} \|A_{s_2}\|^p\|A_{s_1}\|^p\\
      & = \sum_{u \in V} \cprkv{k_2}{u} \sum_{s_1 \in \Arcsin_{k_1}(u),s_1(1)=v} \|A_{s_1}\|^p\\
      & \leq a_{k_2} \sum_{u \in V} \sum_{s_1 \in \Arcsin_{k_1}(u),s_1(1)=v} \|A_{s_1}\|^p\\
      & \leq \cprkv{k_1}{v} a_{k_2}
    \end{align*}
    hence, in particular, $a_k \leq a_{k_1} a_{k_2}$ and $\log a_k \leq \log a_{k_1} + \log a_{k_2}$.
    We can conclude by \lemref{fekete}.
  \end{proof}
\end{lemma}

\begin{corollary}
  \label{coro:summax}
  The following holds
  \[ \lim_{k \to \infty} \left[\max_{v \in V} \cprkv{k}{v}\right]^{\frac{1}{k}} = \lim_{k \to \infty} \left[\sum_{v \in V} \cprkv{k}{v}\right]^{\frac{1}{k}} \]
  and, in particular, the limit on the right-hand side converges.

  \begin{proof}
    For a finite set of nonnegative numbers, their maximum is always between their average and their sum:
    \[ \frac{1}{|V|} \sum_{v \in V} \cprkv{k}{v} \leq \max_{v \in V} \cprkv{k}{v} \leq \sum_{v \in V} \cprkv{k}{v} \]
    or equivalently
    \[ \max_{v \in V} \cprkv{k}{v} \leq \sum_{v \in V} \cprkv{k}{v} \leq |V| \max_{v \in V} \cprkv{k}{v}. \]
    By \lemref{feketemax}, $\max_{v \in V} \cprkv{k}{v}$ converges for $k \to \infty$ hence $\sum_{v \in V} \cprkv{k}{v}$ converges too.
    Taking the $k$th root and the limit $k \to \infty$ gives the identity.
  \end{proof}
\end{corollary}

\begin{lemma}
  \label{lem:rhoAG}
  Consider \defAG{}.
  The following relation holds
  \[
    \cpr = \rho(A(\G)) \lim_{k \to \infty}
    \left[
      \sum_{s \in \Arcs_k} \|A_s\|^p
    \right]^{\frac{1}{pk}}.
  \]
  \begin{proof}
    By \cororef{summax},
    \[
      \lim_{k \to \infty}
      \left[
        \sum_{s \in \Arcs_k} \|A_s\|^p
      \right]^{\frac{1}{pk}}
    \]
    converges.
    It remains to show that $\lim_{k \to \infty} |\Arcs_k|^{-\frac{1}{pk}}$ converges to $\rho(A(\G))$.

    Consider the matrix norm $\|\cdot\|_\infty$ on $\R^{n \times n}$ induced by the infinity norm on $\R^n$.
    It is well known that
    \[ \|A\|_{\infty} = \max_{1 \leq i \leq n} \sum_{j=1}^n |a_{ij}| \]
    where $a_{ij}$ is $(i,j)$ entry of $A$.
    It is also well known that the $(u,v)$ entry of $A(\G)^k$ gives                        % FIXME can remove
    the number of paths of length $k$ starting at node $u$ and ending at node $v$ in $\G$. % FIXME can remove
    Hence $\|A(\G)^k\|_\infty = \mdout_k(\G)$.
    By Gelfand's formula,
    \begin{equation}
      \label{eq:rhoAG}
      \rho(A(\G)) = \lim_{k \to \infty} \|A(\G)^k\|_\infty^{1/k}.
    \end{equation}
    Since $|E_k|/|V| \leq \mdout_k(\G) \leq |E_k|$,
    and $|V|^{1/k} \to 1$ as $k \to \infty$, we are done.
  \end{proof}
\end{lemma}

%In the unconstrained case we know that the $p$-radius is increasing in
%$p$ \cite{zhou2002p}.  We can generalize this fact as follows.
% \begin{lemma}
%   \label{lem:incp}
%   Consider \defAG{}.
%   For any integers $p \leq q$,
%   \[ \cpr[p] \leq \cpr[q] \leq \cjsr \leq \rho(A(\G))^{\frac{1}{q}}
%   \cpr[q] \leq \rho(A(\G))^{\frac{1}{p}} \cpr[p]. \]
% \end{lemma}

\begin{proof}[Proof of \lemref{incp}]
  The Lemma is a consequence of the inequality between ordinary means
  and the inequality between the $p$-norms ($p \geq 1$)
  \[ \|x\|_p = \Big(\sum_{i=1}^n |x_i|^p\Big)^{\frac{1}{p}}. \]
\end{proof}

\begin{lemma}[\cite{hardy1952inequalities}]
  For any nonnegative integers $a_1, \ldots, a_n$ and positive
  real numbers $p \leq q$,
  \[
    \left(\frac{1}{n}\sum_{i=1}^n a_i^p\right)^{\frac{1}{p}}
    \leq
    \left(\frac{1}{n}\sum_{i=1}^n a_i^q\right)^{\frac{1}{q}}
    \leq \max\{\, a_i \mid i = 1, \ldots, n\,\}.
  \]
\end{lemma}

\begin{lemma}
  For any real numbers $1 \leq p \leq q$,
  \[
    \|x\|_\infty \leq \|x\|_q \leq \|x\|_p.
  \]
\end{lemma}

% The purpose of this choice is to have the following lemma.
% \begin{lemma}[{\cite[Lemma~3.5]{philippe2016stability}}]
%   \label{lem:Ad}
%   Consider \defAG{} and a positive integer $d$.
%   The following identity holds:
%   \[ \rho(\G, \A^{[d]}) = \cjsr^d. \]
% \end{lemma}

Let $x^{[d]}$ denote the \emph{scaled monomial} basis.
The elements of this basis are
\[ \frac{d!}{\alpha_1!\alpha_2!\cdots\alpha_n!} x_1^{\alpha_1} \cdots x_n^{\alpha_n}. \]
for each \ktups{n} of nonnegative integers $\alpha$ such that
$\alpha_1 + \cdots + \alpha_n = d$.
For this basis, $\|x^{[d]}\|_2 = \|x\|_2^d$ where $\|\cdot\|_2$ is the Euclidean norm.

For any matrix $A \in \R^{n \times n}$, the map $x \mapsto x^{[d]}$ induces an
associated map $A^{[d]} \in \R^{N_d \times N_d}$ which is the unique
matrix that satisfies $(Ax)^{[d]} = A^{[d]}x^{[d]}$.
We also denote $\A^{[d]} \eqdef \{A_1^{[d]}, \ldots, A_m^{[d]}\}$.

Since $\|Ax\|^{[d]} = \|A\|^{[d]}\|x\|^{[d]}$, we have the following
Lemma that is known in the unconstrained case or for the contrained
case with $p=\infty$.
\begin{lemma}
  \label{lem:pp}
  Consider \defAG{}, then
  \[ \cpr[p] = \rho_1(\G,\A^{[p]})^{\frac{1}{p}} \]
  and
  \[ \cjsr = \rho(\G,\A^{[p]})^{\frac{1}{p}}. \]
\end{lemma}

We say that a cone $\Kset$ is \emph{proper} if it is closed, solid, convex and pointed.
We say that a matrix $A$ \emph{leaves a set $S$ invariant} if $AS \subseteq S$ and
we say that a set of matrices $\A$ \emph{leaves a proper cone invariant} if there
exists a proper cone $\Kset$ such that each matrix of $\A$ leaves $\Kset$ invariant.

\begin{lemma}[{\cite{blondel2005computationally,protasov1997generalized}}]
  \label{lem:unconstrainedinv}
  If a set of $m$ matrices leaves a proper cone $\Kset$ invariant, then
  \begin{align*}
    \rho_1(\A)
    & = \frac{1}{m}\lim_{k\to\infty} \left\|\sum_{s \in [m]^k}A_s\right\|^{\frac{1}{k}}\\
    & = \frac{1}{m}\rho\left(\sum_{A \in \A}A\right).
  \end{align*}
\end{lemma}

% This lemma is a consequence of the following lemma.
% \begin{lemma}[{\cite[Lemma~2]{blondel2005computationally}}]
%   \label{lem:normK}
%   Associated to any proper cone $K$ there is a matrix $\|\cdot\|_K$ that satisfies
%   $\|A\|_K\leq\|A+B\|_K$ for all matrices $A$ and $B$ that leave the cone $K$ invariant.
% \end{lemma}

We deduce the following corollary of \lemref{pp} and
\lemref{unconstrainedinv}.
\begin{corollary}
  If $\A^{[p]}$ leaves a proper cone $\Kset$ invariant, then
  \begin{align*}
    \rho_p(\A)
    & = \frac{1}{m^{\frac{1}{p}}}\lim_{k\to\infty} \left\|\sum_{s \in [m]^k}A_s^{[p]}\right\|^{\frac{1}{pk}}\\
    & = \frac{1}{m^{\frac{1}{p}}}\rho\left(\sum_{A \in \A}A^{[p]}\right)^{\frac{1}{p}}.
  \end{align*}
\end{corollary}

% In view of the definition of $\rho_p$, this corollary seems to say
% that if $\A^{[p]}$ leaves a proper cone invariant,
% taking the norm of each $A_s^{[p]}$ or taking the norm of the sum does
% not make any difference.

We generalize it to the constrained case using the lifting procedure
introduced independantly by Kozyakin~\cite{kozyakin2014berger} and Wang~\cite{wang2014stability}.
% Before that we provide the following lemma.
% \begin{lemma}
%   Given
%   Consider \defAGne{}.
%   Consider a vector norm $\|\cdot\|$ of $\R^n$ and
%   the vector norm $\|\cdot\|'$ of $\R^{n|V|}$ such that
%   \[ \|e_1 \kron x_1 + \cdots + e_{|V|} \kron x_{|V|}\| = \|x_1\| + \cdots + \|x_{|V|}\|. \]
%   Consider the induced matrix norms $\|\cdot\|$ and $\|\cdot\|'$.
% % For any $k$, the following holds
% % \[ \sum_{s \in E_k} \|A_s\| = \sum_{s \in E_k} \|A_s\|' \]
% \end{lemma}

\begin{lemma}
  \label{lem:koz}
  Consider \defAGe{}.
  The following identity holds for any $p \in [1, +\infty]$
  \[
    \sum_{s \in \Arcs_k} \|A_s\|^p =
    \sum_{s \in \Arcs^k} [\|A_s'\|']^p
  \]
  where
  \[ \A' = \{\, A'_{uv\sigma} = (e_ve_u^\Tr) \kron A_\sigma \mid (u,v,\sigma) \in E \,\}. \]
  \begin{proof}
    Consider a vector norm $\|\cdot\|$ of $\R^n$ and
    the vector norm $\|\cdot\|'$ of $\R^{n|V|}$ such that
    \[ \|e_1 \kron x_1 + \cdots + e_{|V|} \kron x_{|V|}\| = \|x_1\| + \cdots + \|x_{|V|}\|. \]
    Consider the induced matrix norms $\|\cdot\|$ and $\|\cdot\|'$.
    It is easy to see that for any nodes $u,v \in V$ and any matrix $B \in \R^{n \times n}$,
    $\|(e_ve_u^\Tr) \kron B\|' = \|B\|$.
    In particular, given a path $s \in \Arcs_k$,
    \begin{align*}
      \|A_s'\|' & = \left\|\prod_{i=1}^k(e_{s(i+1)}e_i^\Tr) \kron A_{s[i]}\right\|'\\
      & = \|(e_{s(k+1)}e_{s(1)}^\Tr) \kron A_s\|' = \|A_s\|
    \end{align*}
    and given $s \notin \Arcs_k$, $\|A_s'\| = 0$.
%   Consider the matrix norm $\|\cdot\|_K$ of \lemref{normK}
%   For any $k$, by subadditivity of the norm,
%   \[ \left\|\sum_{s \in E_k} A_s\right\|_K \leq \sum_{s \in E_k} \|A_s\|_K. \]
%   For any induced norm $\|\cdot\|$, have the following
%   \[ \sum_{s \in E_k} \|A_s\| = \sum_{s \in E_k} \|(e_{s(k+1)}e_{s(1)}^\Tr) \kron A_s\|' \]
%   and by \lemref{normK},
%   \[ |E_k|\left\|\sum_{s \in E_k} A_s\right\|_K \geq \sum_{s \in E_k} \|(e_{s(k+1)}e_{s(1)}^\Tr) \times A_s\|_K.  \]
  \end{proof}
\end{lemma}

% We say that a set of matrices $\A$ contrained by an automaton $\G(\Nodes,\Arcs)$ \emph{leaves proper cones invariant} if there
% exists a proper cone $K_v$ at each node $v \in \Nodes$ such that each edge $(u,v,\sigma) \in \Arcs$, $A_\sigma K_u \subseteq K_v$.
%It is easy to see that if the constrained matrices leave the proper cones $K_v$ invariant then
%the set of matrices $\A'$ of \lemref{koz} leaves the proper cone $\prod_{v \in \Nodes} K_v$ invariant.
It is easy to see that if $\A$ leaves the proper cone $\Kset$ invariant then
the set of matrices $\A'$ of \lemref{koz} leaves the proper cone $\Kset^{|\Nodes|}$ invariant.

\begin{lemma}
  \label{lem:constrainedinv}
  Consider \defAGe{}.
  If $\A$ leaves a proper cone invariant, then
  \begin{align*}
    \cpr[1]
    & = \frac{1}{[\rho(A(G))]^{\frac{1}{p}}}\lim_{k\to\infty} \left\|\sum_{s \in \Arcs_k}(e_{s(k+1)}e_{s(1)}^\Tr) \kron A_s\right\|^{\frac{1}{k}}\\
    & = \frac{1}{[\rho(A(G))]^{\frac{1}{p}}}\rho\left(\sum_{(u,v,\sigma) \in E}(e_ve_u^\Tr) \kron A_\sigma\right).
  \end{align*}
  \begin{proof}
    Combine \lemref{rhoAG}, \lemref{koz} and \lemref{unconstrainedinv}.
  \end{proof}
\end{lemma}

\begin{theorem}
  \label{theo:norminout}
  Consider \defAG{}.
  If $\A^{[p]}$ leaves a proper cone invariant,
  the following identities hold
  \begin{align*}
    \cpr
    & = \frac{1}{[\rho(A(G))]^{\frac{1}{p}}}\lim_{k\to\infty} \left\|\sum_{s \in \Arcs_k}(e_{s(k+1)}e_{s(1)}^\Tr) \kron A_s^{[p]}\right\|^{\frac{1}{pk}}\\
    & = \frac{1}{[\rho(A(G))]^{\frac{1}{p}}}\rho\left(\sum_{(u,v,\sigma) \in E}(e_ve_u^\Tr) \kron A_\sigma^{[p]}\right)^{\frac{1}{p}}.
  \end{align*}
% \begin{equation*}
%   \lim_{k \to \infty}
%   \left[
%     |\Arcs_k|^{-1} \sum_{s \in \Arcs_k} \|A_s^{[p]}\|
%   \right]^{\frac{1}{pk}}
%   = \lim_{k \to \infty}
%   \left[
%     |\Arcs_k|^{-1} \sum_{s \in \Arcs_k} \|A_s\|^p
%   \right]^{\frac{1}{pk}}
%   = \cpr
%   = \lim_{k\to\infty} \left\||\Arcs_k|^{-1} \sum_{s \in \Arcs_k}A_s^{[p]}\right\|^{\frac{1}{pk}}
% \end{equation*}
% or equivalently
% \begin{equation*}
%   \lim_{k \to \infty}
%   \left[
%     \sum_{s \in \Arcs_k} \|A_s^{[p]}\|
%   \right]^{\frac{1}{pk}}
%   = \lim_{k \to \infty}
%   \left[
%     \sum_{s \in \Arcs_k} \|A_s\|^p
%   \right]^{\frac{1}{pk}}
%   = \rho(A(\G))\cpr
%   = \lim_{k\to\infty} \left\|\sum_{s \in \Arcs_k}A_s^{[p]}\right\|^{\frac{1}{pk}}.
% \end{equation*}
\end{theorem}

%The first identity is by choice of the basis, the second one is by
%definition and we will prove the third one in a latter section or in
%appendix (TODO).

Theorem~\ref{theo:norminout} shows that when there is an invariant proper cone,
$\cpr$ is as easy to obtain as computing a spectral radius.
%Indeed, as we will see in a latter
%section (TODO), $\lim_{k\to\infty} \left\|\sum_{s \in \Arcs_k}A_s^{[p]}\right\|^{\frac{1}{pk}}$
%is easy to compute.

It turns out that if $p$ is even then there exists an invariant proper cone.
\begin{lemma}
  Consider \defAG{}.
  For any positive integer $d$,
  $\A^{[2d]}$ leaves an invariant proper cone.
  Moreover this cones is the cone of SOS polynomials in the
  scaled monomial basis.
  \begin{proof}
    Consider an homogeneous SOS polynomial $p(x)$ of degree $2d$ and
    its coordinates $p$ in the scaled monomial basis. That is,
    $p(x) = \la p,x^{[2d]}\ra$. For any matrix $A$, we have
    \[ \la A^{[2d]}p, x^{[2d]} \ra = \la p, (A^{[2d]})^\Tr x^{[2d]} \ra =
    \la p, (A^\Tr x)^{[2d]} \ra = p(A^\Tr x). \]
    Therefore if $p$ is the coordinate vector of an SOS polynomial
    then $A^{[2d]}p$ is also the coordinate vector of an SOS
    polynomial.
  \end{proof}
\end{lemma}

\end{document}